\documentclass[a4paper, 10pt]{article}
\usepackage{etex}
\usepackage[plainpages=false]{hyperref}
\usepackage{amsfonts,latexsym,rawfonts,amsmath,amssymb,amsthm, mathrsfs, lscape, calligra}
\usepackage{verbatim}
\usepackage{enumerate}
\usepackage[margin=1in]{geometry}

\usepackage[all]{xy}
\usepackage{graphicx,psfrag}
\usepackage{pst-node}
\usepackage{tikz-cd}
\usepackage{authblk}

\usepackage{array, tabularx, color}

\usepackage{setspace}
\newtheorem{theorem}{Theorem}[section]
\newtheorem{corollary}[theorem]{Corollary}
\newtheorem{lemma}[theorem]{Lemma}
\newtheorem{claim}[theorem]{Claim}
\newtheorem{proposition}[theorem]{Proposition}

\newtheorem*{thmC*}{Corollary C}

\newcounter{mtheorem}
\newtheorem{mtheorem}[mtheorem]{Theorem}

\newtheorem{mcor}[mtheorem]{Corollary}

\theoremstyle{remark}
\newtheorem{remark}{Remark}

\theoremstyle{definition}
\newtheorem{definition}{Definition}
\newtheorem{example}{Example}

\numberwithin{equation}{section}

\def \C {\mathbb C}
\def \Z {\mathbb Z}
\def \R {\mathbb R}
\def \D {\mathbb D}

\def \F {\mathcal D}

\def \M {\mathcal M}

\def \P {\mathbb P}
\def \CP {\mathbb{CP}}

\def \T {\mathbb T}

\def \p {\partial}
\def \bp {\bar{\partial}}

\def \O {\mathcal{O}}

\def \Cstar {\mathbb C^*}
\def \Cstarn {(\mathbb C^*)^n}
\def \Cstarsqr {(\mathbb C^*)^2}
\def \std {\text{std}}

\def \p {\partial}
\def \bp {\bar{\partial}}
\def \ddt {\frac{\partial}{\partial t}}

\def \Ric {\text{Ric}}

\def \Pot {\mathcal{P}}

\def \t {\mathfrak{t}}

\begin{document}
\title{Uniqueness of shrinking gradient K\"ahler-Ricci solitons on non-compact toric manifolds}
\date{\today}
\author{Charles Cifarelli\thanks{charles\_cifarelli@berkeley.edu} }

\maketitle
\begin{abstract}
We show that, up to biholomorphism, there is at most one complete $T^n$-invariant shrinking gradient K\"ahler-Ricci soliton on a non-compact toric manifold $M$. We also establish uniqueness without assuming $T^n$-invariance if the Ricci curvature is bounded and if the soliton vector field lies in the Lie algebra $\t$ of $T^n$. As an application, we show that, up to isometry, the unique complete shrinking gradient K\"ahler-Ricci soliton with bounded scalar curvature on $\C\P^{1} \times \C$ is the standard product metric associated to the Fubini-Study metric on $\C\P^{1}$ and the Euclidean metric on $\C$. 
\end{abstract}


\section{Introduction}

A \emph{Ricci soliton} $(M,g,X)$ is a Riemannian manifold $(M,g)$ together with a vector field $X$ satisfying 

\begin{equation} \label{rs}
	\Ric_g + \frac{1}{2}\mathcal{L}_Xg = \frac{\lambda}{2}g
\end{equation}
for $\lambda \in \R$. By a simultaneous rescaling of $X$ and $g$, we can always assume that a Ricci soliton is normalized so that $\lambda \in \{-1, 0, +1\}$. We will always assume that the metric $g$ is complete, which in turn forces the vector field $X$ to be complete \cite{Zha}. A Ricci soliton is said to be gradient if the vector field $X$ is the gradient of a smooth function $f$, usually called the \emph{soliton potential}. In this case the equation becomes

\begin{equation} \label{grs}
	\Ric_g + \nabla_g^2 f = \frac{\lambda}{2} g.
\end{equation}

 If $g$ is a K\"ahler metric on $M$ with K\"ahler form $\omega$, we say that $(M, \omega, X)$ is a \emph{K\"ahler-Ricci soliton} if $\omega$ satisfies the equation
 
\begin{equation} \label{krs}
	\Ric_\omega + \frac{1}{2} \mathcal{L}_X\omega = \lambda \omega,
\end{equation}
where $\Ric_\omega$ is the Ricci form and $\lambda \in \{-1, 0, +1\}$. The coefficients appearing in \eqref{krs} are chosen to be different from those in \eqref{rs}; this choice being more natural from the perspective of the K\"ahler-Ricci flow. Ricci solitons and K\"ahler-Ricci solitons are called \emph{expanding}, \emph{steady}, and \emph{shrinking}, respectively when $\lambda \in \{ -1, 0, +1\}$. In this paper we will only consider shrinking solitons and so we will always assume that $\lambda = 1$. As for Ricci solitons, we say that a shrinking K\"ahler-Ricci soliton is gradient if $X = \nabla_g f$, in which case \eqref{krs} takes the form

\begin{equation} \label{gkrs}
	 \Ric_\omega +  i \p \bp f = \omega.
\end{equation}

Ricci solitons are interesting both from the perspective of canonical metrics and of Ricci flow. On the one hand, they represent one direction in which one can generalize the concept of an Einstein manifold. On compact manifolds, shrinking solitons are known to exist in several situations where there are obstructions to the existence of Einstein metrics; see for example \cite{WZ}. By the maximum principle, there are no nontrivial expanding or steady solitons on compact manifolds. There are many examples on noncompact manifolds, however; see for example \cite{ConDer,ConDerSun,Fut} and the references therein. On the other hand, one can associate to a Ricci soliton a self-similar solution of the Ricci flow, and gradient shrinking Ricci solitons in particular provide models for finite-time Type I singularities along the flow \cite{EMT,Naber}. Even in complex dimension two, however, it is not known which shrinking Ricci solitons arise in this way. From this perspective, it is an important problem to classify shrinking gradient K\"ahler-Ricci solitons in order to better understand the singularity development along the K\"ahler-Ricci flow.

In this paper we study K\"ahler-Ricci solitons on non-compact complex manifolds $M$ under the additional assumption that $M$ is toric. For the purposes of this paper, a \emph{complex toric manifold} is a smooth $n$-dimensional complex manifold $(M,J)$ together with an effective holomorphic action of the complex torus $\Cstarn$. In such a setting there always exists an orbit $U \subset M$ of the $\Cstarn$-action which is open and dense in $M$. Moreover, we always assume that there are only finitely many points which are fixed by the $\Cstarn$-action. The $\Cstarn$-action of course determines the action of the real torus $T^n \subset \Cstarn$, and our main theorem is a uniqueness result for complete shrinking gradient K\"ahler-Ricci solitons which are invariant under this action.
\begin{mtheorem} \label{thm1-1}
Suppose that $(M,J)$ is a non-compact complex toric manifold and that the fixed point set of the $\Cstarn$-action is finite. Then, up to biholomorphism, there is at most one complete $T^n$-invariant shrinking gradient K\"ahler-Ricci soliton $(g, X)$ on $(M, J)$. 
\end{mtheorem}

As we will see, $T^n$-invariance implies that the holomorphic vector field $JX$ associated to the soliton vector field $X$ lies in the Lie algebra $\t$ of the real torus $T^{n}$. There is also a notion of a toric manifold coming purely from symplectic geometry. To distinguish this from the definition above, we say that a \emph{symplectic toric manifold} is an $n$-dimensional symplectic manifold $(M,\omega)$ together with an effective Hamiltonian action of the real torus $T^{n}$. As before, we will always assume in this paper that the fixed point set of the $T^{n}$-action is finite. We remark here that this assumption is non-trivial; see for example \cite[Example 6.9]{KarLer}.

Of course, the intersection of these ideas naturally lies in the realm of K\"ahler geometry. In particular, if $(M,J)$ is a complex toric manifold as above and $\omega$ is the K\"ahler form of a compatible K\"ahler metric $g$ on $M$ with respect to which the real $T^{n}$-action is Hamiltonian, then the symplectic manifold $(M,\omega)$ is naturally a symplectic toric manifold. When $(M,J,\omega)$ is a compact K\"ahler manifold, then the two definitions are equivalent in the following sense. Suppose that $(M, J, \omega)$ admits an effective Hamiltonian and \emph{holomorphic} action of the real $n$-dimensional torus $T^{n}$, so that $(M,\omega)$ in particular carries the structure of a symplectic toric manifold. Then this action can always be complexified to an action of the full complex torus $\Cstarn$, giving $(M,J)$ the structure of a complex toric manifold. This can be done essentially because any vector field on $M$ is complete. Of course in the non-compact setting this is no longer the case, and so it makes sense to ask if Theorem \ref{thm1-1} can be extended to the more general setting of symplectic toric manifolds. We prove this under the additional assumption that the Ricci curvature of $g$ is bounded, i.e. $\sup_{x \in M} |\Ric_g|_g(x) < \infty$.
\newline

\begin{mtheorem} \label{thm1-2}
Suppose that $(M,J)$ is a non-compact complex manifold with $\dim_\C M = n$, together with an effective holomorphic action of a real torus $T^{n}$ with Lie algebra $\t$ and finite fixed point set. Then, up to biholomorphism, there is at most one complete shrinking gradient K\"ahler-Ricci soliton $(g, X)$ on $(M, J)$ with $JX \in \t$ and with bounded Ricci curvature.
\end{mtheorem}
Notice that in this case we do not need to assume that $g$ is $T^{n}$-invariant, only that the Ricci curvature is bounded and $JX \in \t$. In fact, we will see in Section 4 that any K\"ahler-Ricci soliton satifsying these hypotheses is isometric to a $T^{n}$-invariant one. When $M$ is compact, these results are special cases of the general uniqueness theorem of Tian-Zhu \cite{TZ1,TZ2}. The non-compact case is generally much more delicate. Typically, one needs to prescribe the asymptotics of the metric, for example by imposing a fixed model metric at infinity, in order to work in well-behaved function spaces. An important feature of this work is that we do not impose any assumptions on the specific behavior of the metric at infinity. Instead, a generalization of the setup of Berman-Berndtsson \cite{BB} allows us to work with the Ding functional on broadly defined $L^{1}$-type spaces; see Section 3 for details.  For a result of even greater generality in the special case when $M = \C$, see also \cite{WWang}.

As an application of Theorem \ref{thm1-2}, we prove a stronger uniqueness result for the special case of $M = \CP^1 \times \C$. 
\begin{mcor} \label{cor1-3}
	Up to isometry, the standard product of the Fubini-Study and Euclidean metrics is the unique complete shrinking gradient K\"ahler-Ricci soliton with bounded scalar curvature on $\CP^{1} \times \C$. 
\end{mcor}
The point is that, if one has a complete shrinking gradient K\"ahler-Ricci soliton with bounded scalar curvature on $M$, then it suffices to assume that $JX$ lies in the Lie algebra of the standard torus acting on $\CP^{1} \times \C$, in which case Theorem \ref{thm1-2} applies directly. This is achieved in Section 4 by a Morse theoretic argument similar to the one implemented in \cite[Proposition 2.27]{ConDerSun}. 

Other well-known examples of complete shrinking gradient K\"ahler-Ricci solitons include those constructed by Feldman-Ilmanen-Knopf \cite{FIK} on the total spaces of the line bundles $\O(-k) \to \C\P^{n-1}$ for $0 < k < n$, which were recently shown to be the unique such metrics on these manifolds with bounded Ricci curvature in \cite[Theorem E]{ConDerSun}. In fact, there are recent examples of Futaki \cite{Fut} which generalize this construction on the total space of any root of the canonical bundle of a compact toric Fano manifold (see also \cite{FutWang,Yang} for the case where the soliton vector field generates an $S^1$-action). All of the aforementioned examples are toric in the sense that underlying manifold is always a complex toric manifold and the metric is invariant under the action of the corresponding real torus $T^{n}$. As before, we denote the Lie algebra of this fixed real torus by $\t$. As a direct consequence of Theorem \ref{thm1-2}, we have that these are the only examples of shrinking gradient K\"ahler-Ricci solitons on these manifolds with bounded Ricci curvature and with $JX \in \t$. 

\begin{mcor} \label{cor1-4}
	Let $N$ be an $(n-1)$-dimensional toric Fano manifold, $L \to N$ be a holomorphic line bundle such that $L^{p} = K_N$ with $0 < p < n$, and let $M$ denote the total space of $L$. There is a natural action of the real torus $T^{n}$ on $M$, and we denote the Lie algebra by $\t$. Then, up to biholomorphism, there is a unique complete shrinking gradient K\"ahler-Ricci soliton with bounded Ricci curvature and $JX \in \t$ on $M$, namely the one constructed by Futaki in \cite{Fut}.
\end{mcor}

We also study the weighted volume functional $F$ on a complex toric manifold. This was introduced by Tian-Zhu \cite{TZ2} for compact manifolds, and is by definition a convex function on the space $\mathfrak{h}$ of all real holomorphic vector fields on $(M,J)$. As in \cite{TZ2}, the derivative of $F$ at a given holomorphic vector field can be viewed as a generalization of the Futaki invariant. The upshot is that if $(g,X)$ is a complete shrinking gradient K\"ahler-Ricci soliton on $M$, then $JX$ is necessarily the unique critical point of $F$. As a result, the vector field $X$ associated to a complete shrinking gradient K\"ahler-Ricci soliton on $(M,J)$ is unique. It was shown in \cite{ConDerSun} using the Duistermaat-Heckman theorem \cite{DH,DHadd,PW} that $F$ can be defined in the non-compact setting in the presence of a holomorphic $T^k$-action when the metric $g$ has bounded Ricci curvature. More precisely, there is an open cone $\Lambda \subset \t$, comprising those holomorphic vector fields which admit Hamiltonian potentials which are proper and bounded from below, on which $F$ is well-defined. Just as in \cite{PW}, we will see in Section 3 that, in the toric setting, there is a natural identification of $\Lambda$ with a certain open convex cone $C^* \subset \t$ determined by the $\Cstarn$-action on $(M,J)$. Furthermore, any soliton vector field $X$ with $JX \in \t$ necessarily has the property that $JX \in \Lambda$ and is the unique critical point of $F$, which in turn gives uniqueness among all holomorphic vector fields $Y$ with $JY \in \t$ \cite[Theorem D]{ConDerSun}.

We show that on a complex toric manifold, the weighted volume functional $F$ is \emph{proper} on $\Lambda$, and therefore that there exists a unique candidate holomorphic vector field $X$ with $JX \in \t$ that could be associated to a complete shrinking gradient K\"ahler-Ricci soliton. Here we make no assumptions on the curvature. Thus, we recover an analog of \cite[Theorem D]{ConDerSun} when the torus is full-dimensional, without having to assume a Ricci curvature bound; see Theorem \ref{thm4-7} below for the precise statement. 

 The main theorems here also give partial answers to some open questions raised in \cite[Section 7.2]{ConDerSun}. Namely, we obtain a positive answer to question 7 assuming that the torus is the real torus underlying an effective holomorphic and full-dimensional $\Cstarn$-action with finite fixed point set, and a positive answer to question 2 with the same assumption on the torus as well as the assumption that either $g$ is invariant or that $g$ has a Ricci curvature bound. We also show that any symplectic toric manifold with finite fixed point set admitting a compatible complete shrinking gradient K\"ahler-Ricci soliton is quasiprojective, which gives a positive answer to question 1 in the toric setting. Finally, we show that the weighted volume functional $F$ is proper on a complex toric manifold with finite fixed point set, which gives a positive answer to question 9 when the real torus is full-dimensional and admits a complexification. As we will see, this is always the case in the presence of an invariant solution to \eqref{gkrs}. 

Since the foundational work of Delzant \cite{Del} and Guillemin \cite{Guil} (which themselves relied on the earlier foundational work of Atiyah \cite{Atiyah} and Guillemin-Sternberg \cite{GS}), toric manifolds have played a key role in the study of special K\"ahler metrics on compact K\"ahler manifolds; see \cite{Ab1,DonStabTor,WZ} and many others. As a consequence of this setup, many aspects of the K\"ahler geometry of $T^n$-invariant metrics on $M$ reduce to questions about convex functions on a given polytope $P$ in $\R^n$. We show that under certain mild hypotheses, much of the structure from the compact setting carries over, replacing the bounded polytopes with potentially unbounded polyhedra. In the purely symplectic setting, there has been much work done in this direction, spanning many years; see \cite{Abook,HNP,KarLer,Ler,PW}. There has been somewhat less attention focused on the K\"ahler case, and our work draws significantly on the notable exceptions of \cite{AbSD,BGL,MSY,VC1}. There has also been recent progress in the K\"ahler setting on \emph{singular} toric varieties; see \cite{BGL} and of particular relevance to this paper \cite{BB}. 

The paper is organized as follows. In Section 2 we recall some of the basics of toric geometry from both the algebraic and symplectic perspectives. We show that the Abreu-Guillemin setup can be extended with the appropriate assumptions to non-compact manifolds. Much of this material seems to be fairly well-known in the symplectic setting, and we simply provide a rephrasing particularly suited for K\"ahler geometry. In particular, we give conditions under which the familiar Delzant classification holds in the non-compact setting. In Section 3 we study properties of some real Monge-Amp\`ere equations on unbounded convex domains in $\R^n$, and explain how these relate to the K\"ahler-Ricci soliton equation on toric manifolds. We introduce a Ding-type functional $\F$ on the appropriate space of symplectic potentials and use its convexity to determine uniqueness. Much of what appears here is drawn from \cite{BB} and \cite{Don1}. A result of Wylie \cite{Wylie} implies that any complete shrinking gradient K\"ahler-Ricci soliton admits a moment map. In Section 4, we use this to apply the results of the previous sections to complete the proofs of Theorem \ref{thm1-1} and Theorem \ref{thm1-2}.  We also include in Section 4 a proof of Corollary \ref{cor1-4}, which amounts to demonstrating that the examples constructed in \cite{Fut} indeed have bouned Ricci curvature. We conclude with an application of our work to the special case of $M = \CP^{1} \times \C$, and show that a complete shrinking gradient K\"ahler-Ricci soliton on $M$ is isometric to the standard product metric. This is the content of Corollary \ref{cor1-3}.

\subsection*{Acknowledgements}
I would like to thank Song Sun for suggesting this problem, as well as for his continued support and many useful discussions. I would like to thank Ronan Conlon for his interest, useful discussions, and detailed commentary on the preliminary versions of this article. I would also like to thank Vestislav Apostolov for his invaluable suggestions and Akito Futaki for his interest and comments. I am also grateful for the many conversations with Chris Eur in which I learned much of what I know about the algebraic geometry of toric varieties, and to the referee for their insightful suggestions. This work was partially supported by the National Science Foundation RTG grant DMS-1344991.

\section{K\"ahler geometry on non-compact toric manifolds}

\subsection{Algebraic preliminaries}

We begin by recalling some basics from algebraic toric geometry that we will use later on. The main reference here is \cite{CLS}. Fix an algebraic torus $\Cstarn$ and let $\t$ be the Lie algebra of the real torus $T^n \subset \Cstarn$. Fix an integer lattice $\Gamma \subset \t$ so that $\Cstarn \cong \t \oplus i \t / \Gamma$ acting only in the second factor. Let $\Gamma^*$ denote the corresponding dual lattice in $\t^*$. 

\begin{definition} \label{defi2-1}
 A \emph{toric variety} $M$ is an algebraic variety together with the effective algebraic action of the complex torus $\Cstarn$ with a dense orbit. More precisely, this means that the action $\Cstarn \times M \to M$ is a morphism of algebraic varieties, and there exists a point $p \in M$ such that the orbit $\Cstarn \cdot p \subset M$ is Zariski open and dense in $M$. 
\end{definition}

We emphasize that, contrary to the definitions presented in the introduction, a toric \emph{variety} $M$ is always assumed to be algebraic. As we will see, the fixed point set of the $\Cstarn$-action associated to a toric variety is necessarily finite. In particular, the underlying complex manifold of a smooth toric variety is always a complex toric manifold as defined in the introduction.

 The algebraic geometry of toric varieties has a rich interplay with combinatorics, which is integral to many of the constructions that follow. We begin by introducing the relevant combinatorial objects.

\begin{definition}\label{defi2-2}
	A \emph{polyhedron} is any finite intersection of affine half spaces $H_{\nu,a} =\{x \in \t^* \: | \:  \langle \nu, x \rangle \geq a \}$ with $\nu \in \t, a \in \R$. A \emph{polytope} is a bounded polyhedron.
\end{definition}

We will often not distinguish between a polyhedron $P$ and its interior, but where confusion may arise we will denote by $\overline{P}$ the closed object and $P$ the interior. The intersection of $P$ with the plane $\langle \nu, x \rangle = a$ is a polyhedron $F_\nu$ of one less dimension and is called a \emph{facet} of $P$. The intersections of any number of the $F_\nu$'s form the collection of \emph{faces} of $P$. 

\begin{definition} \label{defi2-3}
Let $P$ be a polyhedron given by the intersection of the half spaces $H_{\nu_i, a_i}$. We define the \emph{recession cone} (or asymptotic cone) $C$ of $P$ by 

\begin{equation*}
	C = \left\{ x \in \t^* \: | \: \langle \nu_i, x \rangle \geq 0 \right\}.
\end{equation*}
\end{definition}
\noindent Given any convex cone $C \subset \t$, the \emph{dual cone} $C^* \subset \mathfrak{t}$ is defined by

\begin{equation} \label{dualcone}
	C^* = \{\xi \in \mathfrak{t} \: | \: \langle \xi, x \rangle \geq 0 \text{ for all } x \in C\}.
\end{equation}
Note that (the interior of) $C^*$ is necessarily an open cone in $\mathfrak{t}$, even when $C$ is not full-dimensional.

\begin{definition} \label{defi2-4}
	Let $P$ be a polyhedron. If the vertices of $P$ lie in the dual lattice $\Gamma^* \subset \t^*$, then we say that $P$ is \emph{rational}.
\end{definition}

Rational polyhedra play an important role in the algebraic geometry of toric varieties, in that each such $P$ determines a unique quasiprojective toric variety $\M_P$. This procedure is constructive and can be understood via the introduction of a \emph{fan}. A \emph{rational polyhedral cone} $\sigma$ is by definition a convex subset of $\t$ of the form 

\begin{equation*}
	\sigma = \left\{ \sum \lambda_i \nu_i \: | \: \lambda_i \in \R_+ \right\},
\end{equation*}
where $\nu_1, \dots, \nu_k \in \Gamma$ is a fixed finite collection of lattice points. The recession cone $C$ of a rational polyhedron $P$ is always a rational polyhedral cone \cite[Chapter 7]{CLS}.

\begin{definition} \label{defi2-5}
A \emph{fan} $\Sigma$ in $\t$ is a finite set consisting of rational polyhedral cones $\sigma$ satisfying

	\begin{enumerate} 
		\item For every $\sigma \in \Sigma$, each face of $\sigma$ also lies in $\Sigma$.
		\item For every pair $\sigma_1, \sigma_2 \in \Sigma$, $\sigma_1 \cap \sigma_2$ is a face of each.
	\end{enumerate}	
\end{definition}
\noindent We will also assume that the support of $\Sigma$ is full-dimensional, that is to say, there exists at least one $n$-dimensional cone $\sigma \in \Sigma$. To every fan $\Sigma$ there is an associated toric variety $\M_\Sigma$. We will give a very brief summary of this construction below; for more details see \cite[Chapter 3]{CLS}. For us the main point is the following corollary of a result of Sumihiro \cite{Sumi}: 
 
 \begin{proposition}[{\cite[Corollary 3.1.8]{CLS}}]\label{prop2-6}
 	Let $M$ be a toric variety. Then there exists a fan $\Sigma$ such that $M \cong \M_\Sigma$. 
 \end{proposition}
To construct $\M_\Sigma$ from $\Sigma$, one begins by taking each $n$-dimensional cone $\sigma \in \Sigma$ and constructing an affine toric variety $U_\sigma$. We define the dual cone $\sigma^*$ of $\sigma$ by \eqref{dualcone}:

	\begin{equation*}
		\sigma^* = \left\{x \in \t^* \: | \: \langle x, \xi \rangle \geq 0 \text{ for all } \xi \in \sigma \right\}.
	\end{equation*}
Let $S_\sigma$ be the semigroup of those lattice points which lie in $\sigma^*$ under addition. Then one defines the semigroup ring, as a set, as all finite sums of the form 

	\begin{equation*}
		\C[S_\sigma] = \left\{ \sum \lambda_s s \: | \: s \in S_\sigma \right\}.
	\end{equation*}
The ring structure is then defined on monomials by $\lambda_{s_1}s_1\cdot \lambda_{s_2}s_2  = (\lambda_{s_1}\lambda_{s_2})(s_1+ s_2)$ and extended in the natural way. The basic example is $\sigma = \R^n_+$, where $\C[S_\sigma]$ is naturally isomorphic to $\C[z_1, \dots, z_n]$. Then the affine variety $U_\sigma$ is defined to be $\text{Spec}(\C[S_\sigma])$. This is automatically endowed with a $\Cstarn$-action with an open dense orbit. This construction of course can be implemented on the lower-dimensional cones $\tau \in \Sigma$. If $\sigma_1 \cap \sigma_2 = \tau$, then there is a natural way to map $U_\tau$ into $U_{\sigma_1}$ and $U_{\sigma_2}$ isomorphically. Thus one constructs $\M_\Sigma$ by declaring the collection of all $U_\sigma$ to be an open affine cover with transition data determined by $U_\tau$. An important property of this construction is the Orbit-Cone correspondence. 

 \begin{proposition}[Orbit-Cone correspondence,{ \cite[Theorem 3.2.6]{CLS}}] \label{prop2-7} Let $\Sigma$ be a fan and $\M_\Sigma$ be the associated toric variety. The $k$-dimensional cones $\sigma \in \Sigma$ are in natural one-to-one correspondence with the $(n-k)$-dimensional orbits $O_\sigma$ of the $\Cstarn$-action on $\M_\Sigma$. Moreover, given a $k$-dimensional cone $\sigma \in \Sigma$ and a corresponding orbit $O_\sigma \subset \M_\Sigma$, we have that $\sigma$ lies as an open subset of the Lie algebra $\t_\sigma$ of the $k$-dimensional real subtorus $T_\sigma \subset T^n$ that stabilizes the points on $O_\sigma$.
 \end{proposition}
In particular, the fixed point set of the $\Cstarn$-action is in natural bijection with the full-dimensional cones in $\Sigma$, and is therefore always finite. At the other extreme, each ray $\sigma \in \Sigma$ determines a unique torus-invariant divisor $D_\sigma$. As a consequence, a torus-invariant Weil divisor $D$ on $\M_\Sigma$ naturally determines a polyhedron $P_D \subset \t^*$ as follows. We can decompose $D$ uniquely as $D = \sum_{i = 1}^{N} a_i D_{\sigma_i}$, where $\sigma_i \in \Sigma$, $i = 1, \dots, N$ is the collection of rays. By assumption, there exists a unique minimal $\nu_i \in \sigma_i \cap \Gamma$. Then set 
 
 \begin{equation} \label{polyofdivisor}
 	P_D = \left\{ x \in \t^* \: | \: \langle \nu_i, x \rangle \geq - a_i  \text{ for all } i = 1, \dots, N \right\}.
 \end{equation}
The importance of polyhedra for our purposes lies in the fact that this procedure is partially reversible. That is, given a suitable polyhedron $P$, one can determine a unique toric variety $\M_P$ through its \emph{normal fan} $\Sigma_P$. To form $\Sigma_P$, one starts with a vertex $v \in P$ and considers those facets $F$ containing $v$. This determines a cone $\sigma_v$ spanned by the inner normals $\nu_F$ corresponding to each such $F$. Then there is a unique fan $\Sigma_P$ which consists of the collection of $\sigma_v$ along with all of each of their faces. Finally, $\M_P$ is defined to be the toric variety associated to $\Sigma_P$. As we will see, the variety $\M_P$ comes naturally equipped with a divisor $D$ whose corresponding polyhedron is precisely $P$. Moreover, 

\begin{proposition}[{\cite[Theorem 7.1.10]{CLS}}] \label{prop2-8}
	Let $P$ be a full-dimensional rational polyhedron in $\t^*$. Then the variety $\M_P$ constructed above is quasiprojective.
\end{proposition}

\subsection{Complex coordinates}

Let $M$ be a complex manifold together with an effective holomorphic $\Cstarn$-action. Such an action always has an open and dense orbit. Indeed, let $T^n \subset \Cstarn$ be the real torus with Lie algebra $\t$. Choose a basis $(X_1, \dots, X_n)$ for $\t$. Then each $X_i$ is a holomorphic vector field on $M$, and thus vanishes along an analytic subvariety. In particular, there is a fixed analytic subvariety $V \subset M$ such that on $U = M - V$, none of the vector fields $X_i$ vanish. Clearly $X_i$ and $JX_i$ are complete and commute, and so the vector fields $(X_1, JX_1, \dots, X_n, JX_n)$ can be integrated to determine an isomorphism $U \cong \Cstarn$. Throughout the remainder of the paper we will make heavy use of this natural coordinate system, which we usually just denote by $\Cstarn \subset M$. In particular, we fix once and for all such a basis $(X_1, \dots, X_n)$ for $\t$. This induces a background coordinate system $(\xi^1, \dots, \xi^n)$ on $\t$. We use the natural inner product on $\t$ to identify $\t \cong \t^*$ and thus can also identify $\t^* \cong \R^n$. For clarity, we will denote the induced coordinates on $\t^*$ by $(x^1, \dots, x^n)$. Let $(z_1, \dots, z_n)$ be the natural coordinates on $\Cstarn$ as an open subset of $\C^n$. There is a natural diffeomorphism $\text{Log}: \Cstarn \to \t \times T^n$, which provides a one-to-one correspondence between $T^n$-invariant smooth functions on $\Cstarn$ and smooth functions on $\t$. Explicitly, $\text{Log}(z_1, \dots, z_n) = (\log(r_1), \dots, \log(r_n), \theta_1, \dots, \theta_n)$, where $z_j = r_j e^{i \theta_j}$. Given a function $H(\xi)$ on $\t$, we can extend $H$ trivially to $\t \times T^n$ and pull back by Log to obtain a $T^n$-invariant function on $\Cstarn$. Clearly, any $T^n$-invariant function on $\Cstarn$ can be written in this form.

\begin{definition} \label{defi2-9}
	Let $\omega$ be a $T^{n}$-invariant K\"ahler metric on $M$. We say that the $T^{n}$-action is \emph{Hamiltonian} with respect to the $\omega$ if there exists a \emph{moment map} $\mu$. This by definition is a smooth function $\mu: M \to \mathfrak{t}^*$ satisfying

\begin{equation*} 
	d\langle \mu, v \rangle = - i_v\omega,
\end{equation*} 
for each $v \in \mathfrak{t}$ where $i_v$ denotes the interior product and $\langle \cdot \, ,  \cdot \rangle$ is the dual pairing. 
\end{definition}
The K\"ahler metrics on the complex torus $\Cstarn$ itself with respect to which the standard $T^{n}$-action is Hamiltonian have a natural characterization due to Guillemin.

\begin{proposition}[{\cite[Theorem 4.1]{Guil}}]\label{prop2-10} 
	Let $\omega$ be any $T^n$-invariant K\"ahler form on $\Cstarn$. Then the action is Hamiltonian with respect to $\omega$ if and only if there exists a $T^n$-invariant potential $\phi$ such that $\omega =2 i\p\bp \phi$. 
\end{proposition}

Suppose that $(M, J, \omega)$ admits an effective and holomorphic $\Cstarn$-action and that $\omega$ is the K\"ahler form of a $T^n$-invariant compatible K\"ahler metric. In this context, Proposition \ref{prop2-10} implies that if the $T^{n}$-action on $M$ is Hamiltonian with respect to $\omega$, then restriction of $\omega$ to the dense orbit is $\p\bp$-exact. As before, let $(z_1, \dots, z_n)$ denote the standard coordinates on $\Cstarn$. Choose any branch of $\log$ and write $w = \log(z)$. Then clearly $w = \xi + i \theta$ (or, more precisely, there is a corresponding lift of $\theta$ to the universal cover with respect to which the equality holds), and so if $\phi$ is $T^n$-invariant and $\omega = 2i \p \bp \phi$, we have that

\begin{equation*}
	\omega = 2i\frac{\p^2 \phi}{ \p w^i \p\bar{w}^j} dw_i \wedge d\bar{w}_j = \frac{\p^2 \phi}{ \p \xi^i \p\xi^j} d\xi^i \wedge d\theta^j.
\end{equation*}
In this setting, the metric $g$ corresponding to $\omega$ is given on $\t \times T^n$ by 

\begin{equation*}
	g = \phi_{ij}(\xi)d\xi^i d\xi^j + \phi_{ij}(\xi)d\theta^i d\theta^j.
\end{equation*}
The moment map $\mu$ as a map $\mu: \t \times T^n \to \t^*$ is defined by the relation 
	
	\begin{equation*}
		\langle \mu(\xi, \theta), b \rangle = \langle \nabla \phi(\xi), b \rangle
	\end{equation*}
for all $b \in \t$, and where $\nabla \phi$ is the Euclidean gradient of $\phi$. Since the Hessian of $\phi$ is positive-definite, it follows that $\phi$ is strictly convex on $\t$. In particular, $\nabla \phi$ is a diffeomorphism onto its image. Using the identifications mentioned at the beginning of this section, we view $\nabla \phi$ as a map from $\t$ into an open subset of $\t^*$.

\subsection{Setup of the equation}

Suppose now that $(g,X)$ is a shrinking gradient K\"ahler-Ricci soliton on a complex toric manifold $M$ and that $g$ is $T^n$-invariant. Restricting to the dense orbit, we see that $g$ is determined by a convex function $\phi$ on $\t$. We wish therefore to write equation \eqref{gkrs} as an equation for $\phi$. From \eqref{gkrs}, we can assume by averaging that the soliton potential $f$, and therefore the vector field $X$, must also be $T^n$-invariant. Writing $f = f(\xi, \theta)$ in the real coordinate system $(\xi, \theta)$ above, it follows that $f$ is independent of $\theta$. Therefore we have that 

	\begin{equation}\label{eqn2-3}
		X = \nabla_g f =  \phi^{ij} \frac{\p f}{\p \xi^i} \frac{\p}{\p \xi^j}.
	\end{equation}
In fact, the coefficients $\phi^{ij}\frac{\p f}{\p \xi^i}$ must be constant. Indeed, let $w = \log(z)$ as above, where $z$ is the standard coordinate on $\Cstarn$, so that $w = \xi + i \theta$. In these coordinates we can write 

\begin{align*} 
	X^{1,0} = \phi^{ij} \frac{\p f}{\p \xi^i}\frac{\p}{\p w_j},
\end{align*}
where the coefficients $ \phi^{ij} \frac{\p f}{\p \xi^i}$ depend only on the real part $\xi$ of $w$. Since $X$ is holomorphic, it follows that 

\begin{align*}
	\frac{\p}{\p \xi^k} \left( \phi^{ij} \frac{\p f}{\p \xi^i}\right)= 2 \frac{\p}{\p \bar{w}_k} \left( \phi^{ij} \frac{\p f}{\p \xi^i}\right) = 0.
\end{align*}
In particular, it follows that $JX \in \t$. We will denote the coefficients $\phi^{ij}\frac{\p f }{\p \xi^j} = b_X^i$, so that $JX = b_X^i \frac{\p}{\p \theta^i}$ is determined by the constant $b_X \in \t$.

 \begin{lemma}  \label{lem2-11}
 	Suppose that $\omega$ is a $T^n$-invariant K\"ahler metric on $M$ and that the $T^n$-action is Hamiltonian with respect to $\omega$, so that there exists a K\"ahler potential $\phi$ for $\omega$ on the dense orbit $\Cstarn \subset M$. If $Y$ is any real holomorphic vector field such that $JY \in \mathfrak{t}$, let $\theta_Y \in C^\infty(M)$ be the Hamiltonian potential $\theta_Y = \mu(JY)$ corresponding to $JY$. Then $\theta_Y$ also satisfies $\mathcal{L}_Y\omega = 2i \p \bp \theta_Y$. Moreover, up to a constant, the restriction of $\theta_Y$ to the dense orbit is given by $\theta_Y(\xi,\theta) = Y(\phi)$.  
 \end{lemma}
 
 \begin{proof}
 	By Cartan's formula it suffices to show that
	 
	\begin{equation*}
		i_Y\omega = -Ji_{JY}\omega = -J d \mu(JY) =  d^c \mu(JY),
	\end{equation*}
which proves the first statement. The second statement follows immediately from the fact that the restriction of $\omega$ to the dense orbit is given by $2i \p\bp \phi$. 	
 \end{proof}
 
On the dense orbit then, the term $\mathcal{L}_X\omega$ in \eqref{gkrs} is given by 

 \begin{equation*}
 	 \mathcal{L}_X\omega = 2i\p\bp X(\phi).
 \end{equation*}
Hence, up to a constant, the soliton potential $f$ is given in real logarithmic coordinates on the dense orbit by

\begin{equation} \label{eqn2-4}
	 f = X(\phi) = b_X^j \frac{\p \phi }{\p \xi^j}.
\end{equation}
 Since the Ricci form of $\omega$ is given by 
 
 	\begin{equation*}
		\Ric_\omega = -i \p \bp \log \det(\phi_{ij}),
	\end{equation*}
we can succinctly rewrite \eqref{gkrs} in terms of $\phi$ alone.

  \begin{proposition}\label{prop2-12}  
  	Suppose that $M$ is a complex toric manifold and $(\omega,X)$ is a shrinking gradient K\"ahler-Ricci soliton. If the $T^n$-action is Hamiltonian with respect to $\omega$, then $\omega$ has a K\"ahler potential $\phi$ on the dense orbit, which can be viewed via the identification $\t \times T^n \cong \Cstarn$ as a convex function on $\R^n$. Then there exists a unique affine function $a(\xi)$ on $\R^n$ such that $\phi_a = \phi - a$ satisfies the real Monge-Amp\`ere equation

\begin{equation} \label{eqn2-5}
	\det (\phi_a)_{ij} = e^{-2\phi_a + \langle b_{X}, \nabla\phi_a\rangle}.
\end{equation}

\end{proposition}

\begin{proof}
	In light of the above discussion, the soliton equation \eqref{krs}
	
	\begin{equation*}
		\omega - \Ric_\omega - \frac{1}{2}\mathcal{L}_X\omega = 0 
	\end{equation*}
	can be rewritten as
	
	\begin{align*}
		0 &= i \p \bp \left( 2\phi + \log\det(\phi_{ij}) - X(\phi) \right) \\
		  &= 2\frac{\p^2 }{\p\xi^i\p\xi^j} \left( 2 \phi + \log\det(\phi_{ij}) - \langle b_X , \nabla\phi \rangle \right) d\xi^i \wedge d\theta^j,
	\end{align*}
	and so the function $ 2 \phi + \log\det(\phi_{ij}) - \langle b_X , \nabla\phi \rangle$ on $\R^n$ has vanishing Hessian, and is therefore equal to an affine function $a(\xi)$. Define 
	
	\begin{equation*}
		\phi_{a}(\xi) = \phi(\xi) - \frac{1}{2} a(\xi)
	\end{equation*}
and let $c$ be the constant $c = \frac{1}{2} \langle b_X, \nabla a \rangle$. Then it is clear that 
	
	\begin{equation*}
		2 \phi_{a} + \log\det(\phi_{a,ij}) - \langle b_X , \nabla\phi_{a} \rangle = c.
	\end{equation*}
Thus, by modifying $a$ by the addition of a constant, we have that $\phi_a$ satisfies \eqref{eqn2-5}.\\
\end{proof}

As we have seen, the metric $g$ depends only on the Hessian of $\phi$. Part of the content of Proposition \ref{prop2-12} therefore is a normalization for the potential $\phi$, and we will make use of this later on.

\subsection{Polyhedra and symplectic coordinates}

\begin{definition} \label{defi2-13}
	Let $P$ be a full-dimensional polyhedron in $\t^*$. Then $P$ is called \emph{Delzant} if, for each vertex $v \in P$, there are exactly $n$ edges $e_i$ stemming from $p$ which can be written $e_i = v + \lambda_i \varepsilon_i$ for $\lambda_i \in \R$ and $(\varepsilon_i)$ a $\Z$-basis of $\Gamma^*$. 
\end{definition}

This says that each vertex of a Delzant polyhedron, when translated to the origin, can be made to look locally like standard $\R^n_+$ via an element of $\text{GL}(n,\Z)$. It follows from the definition that there is a well-defined normal fan $\Sigma$ associated to any Delzant polyhedron $P$. One only needs to check that the relevant cones are rational polyhedral cones. This can be shown by induction, for example, since any face of a Delzant polyhedron must itself be Delzant. Therefore, given any Delzant polyhedron $P$, there is an associated toric variety $\M_P = \M_{\Sigma}$. The condition on the vertices of $P$ is precisely what is required to ensure that $\M_P$ is smooth; see \cite[Theorem 3.1.19]{CLS} and the preceding statements there.

In Section 2.1, we encountered a purely algebraic construction which produced a toric variety, and therefore a complex toric manifold, $\M_P$ from the data of a Delzant polyhedron. We now introduce a \emph{different} construction, this time coming from symplectic geometry, which will produce a symplectic toric manifold from the data of $P$. The idea is to construct a complex symplectic manifold $(M_P, \omega_P, J_P)$ as a K\"ahler quotient of $\C^N$ by a subgroup $G_\C$ of the standard torus $(\C^*)^N$. The next proposition is standard for compact symplectic toric manifolds, and in the more general setting of potentially singular and non-compact varieties it is essentially proved in \cite[Lemma 2.1]{BGL}, and earlier in \cite[Chapter VI, Proposition 3.1.1]{Abook}. We could not find the precise statement that we use in the literature, and so we briefly outline the proof below. 

\begin{proposition} \label{prop2-14}
	Let $P$ be a Delzant polyhedron in $\t$ with $N$ facets. Then there exists a K\"ahler manifold $(M_P, \omega_P, J_P)$ with an effective $J_P$-holomorphic $\Cstarn$-action on $M_P$ associated to $P$, obtained as a K\"ahler quotient of $\C^N$ by a complex subgroup $G \subset (\C^*)^N$ acting in the usual way. The $T^n$-action is Hamiltonian with respect to $\omega_P$, and the moment map $\mu_P: M_P \to \t^*$ has image $\overline{P}$. If $P$ is rational, then $\omega_P$ is the curvature form of a hermitian metric on an equivariant line bundle $L_P \to M_P$ determined by $P$.
\end{proposition}

\begin{proof}
	This is a direct consequence of \cite[Lemma 2.1]{BGL}. In particular there is a complex subgroup $G \subset (\Cstar)^N$, a corresponding maximal compact subgroup $K \subset G$, and a moment map $\mu_K$ for the $K$-action on $\C^N$. Then $M_P$ is defined as the symplectic quotient $Z/K$, where $Z \subset \C^N$ is the preimage of a particular regular value of $\mu_K$. Denote the quotient map by $\pi: Z \to M_P$. The symplectic form $\omega_P$ is induced by the symplectic quotient by restricting the standard Euclidean symplectic form $\omega_E$ to $Z$. The complex structure $J_P$ on $M_P$ is determined via the usual K\"ahler quotient construction. In particular, there is a closed analytic subset $V$ in $\C^N$ where $G$ acts freely, and we can equivalently define $M_P = (\C^N - V)/G$. 
	
	Now, if the vertices of $P$ lie on the integer lattice, then the group $G$ is algebraic and the construction of $M_P$ in \cite{BGL} becomes a GIT quotient (see for example \cite[Chapter 14]{CLS} for details on this point). In particular, $P$ determines a character $\chi_P: G \to \Cstar$ which gives rise to an action of $G$ on the trivial line bundle $\mathcal{O} \to \C^N$, and the quotient of $\mathcal{O}$ by $G$ is a well-defined line bundle $L_P$ on $M_P$ \cite[Theorem 14.2.13]{CLS}.  The fact that $\omega_P \in 2\pi c_1(L_P)$ follows directly from the explicit Guillemin formula \cite[Theorem 5.1]{BGL} for $\omega_P$, \cite[Theorem 14.2.13]{CLS}, and the following proposition, which we state separately below for emphasis. 
\end{proof}

In particular, given the data of a \emph{rational} polyhedron $P$, we have two constructions, each associating to $P$ a toric geometric object in the appropriate category. These turn out, after making the relevant identifications, to be equivalent. Let $P$ be a rational polyhedron and $\M_P$ be the toric variety constructed in Section 2.1. 

\begin{proposition} \label{prop2-15}
	The complex manifold $(M_P, J_P)$ is equivariantly biholomorphic to $\M_P$. 
\end{proposition}

We omit the proof here, but this is essentially proven in \cite[Lemma 2.1]{BGL} (c.f. \cite[Chapter VI, Proposition 3.2.1]{Abook}). From the description of $M_P$ given there, one simply applies the main theorem in \cite{Cox} to deduce the proposition. 

In sum, given the data of a \emph{rational} polyhedron $P$, we have two constructions, each associating to $P$ a toric geometric object in the appropriate category, and these constructions are compatible up to an appropriate identification. For the remainder of this section we work with a given Delzant polyhedron $P$ and denote $M = \M_P \cong M_P$. In particular, we have a canonical K\"ahler metric $\omega_P$ on $M$. The induced $T^n$-action is Hamiltonian by construction, so that by Proposition \ref{prop2-10} there is a K\"ahler potential $\omega_P = 2 i \p \bp \phi_P$ on the dense orbit.

We move on to consider an arbitrary K\"ahler metric $\omega$ on $M$ with respect to which the $T^n$-action is Hamiltonian, not necessarily equal to $\omega_P$. We impose the additional assumption that the corresponding moment map $\mu$ also has image equal to $\overline{P}$. Recall from Proposition \ref{prop2-10} that there then exists a potential $\phi$ on the dense orbit $\Cstarn \subset M$. We introduce logarithmic coordinates $(\xi^j, \theta^j)$ as in the previous section so that the moment map $\mu$ is determined by the diffeomorphism $\nabla \phi: \t \to P$. We can then use the moment map to introduce a change of coordinates $\nabla \phi = x$, and thereby view $\Cstarn \cong  P \times T^n$. In these coordinates the K\"ahler form $\omega$ is standard, i.e.

\begin{equation*} 
	\omega = dx^j \wedge d\theta^j.
\end{equation*} 
So the moment map $\mu = \nabla \phi$ induces a natural choice of Darboux coordinates, and for this reason $(x^j, \theta^j)$ are typically referred to as \emph{symplectic coordinates} on $M$. This is only a real coordinate system, and hence the coefficients of the K\"ahler form do not determine those of the corresponding Riemannian metric. One can still determine the metric $g$ by introducing a smooth function $u$ on $P$ which is related to $\phi$ by the \emph{Legendre transform}:

\begin{equation} \label{eqn2-6}
	\phi(\xi) + u(x) =  \langle \xi, x \rangle.
\end{equation}
Then the metric $g$ is given by

\begin{equation} \label{eqn2-7}
	g = u_{ ij}(x)dx^i dx^j + u^{ij}(x)d\theta_i d\theta_j.
\end{equation} 
Thus the metric structure is determined by the Hessian of the function $u$, and so by analogy with the complex case this function is sometimes called the \emph{symplectic potential} for $g$. Although we will not use this here, it is worth noting that it is more natural to view the function $u$ as determining the complex structure $J$, from which the formula \eqref{eqn2-7} for the metric is a consequence. The Legendre transform will be used heavily in the remainder of the paper, and so for convenience we collect some basic properties here. For references focusing on aspects most closely related to the situation here; see for example \cite{BB,Don1,Guil}. 

\begin{lemma} \label{lem2-16}
	Let $V$ be a real vector space and $\phi$ be a smooth and strictly convex function on a convex domain $\Omega' \subset V$. Then there is a unique function $L(\phi) = u$ defined on $\Omega = \nabla\phi(\Omega') \subset V^*$ by \eqref{eqn2-6}:
	
	\begin{equation*} 
		\phi(\xi) + u(x) = \langle \xi, x \rangle
        \end{equation*}
 for $x = \nabla \phi (\xi)$. The function $u$ is smooth and strictly convex on $\Omega$. Moreover, $L$ has the following properties:

\begin{enumerate}
	\item $L(L(\phi)) = \phi $,
	\item $\nabla \phi: \Omega' \to \Omega$ and $\nabla u: \Omega \to \Omega'$ are inverse to each other,
	\item $\phi_{ij}(\nabla u(x)) = u^{ij}(x) $,
	\item $L((1-t)\phi + t \phi') \leq (1-t)L(\phi) + t L(\phi')$,
	\item $L(\phi)(x) = \sup_{\xi \in \Omega'}\{\langle x, \xi \rangle - \phi(\xi)  \}$.
\end{enumerate}
	
\end{lemma}

The third item can be understood to mean that the Euclidean Hessians $\nabla^2\phi$ and $\nabla^2 u$ are inverse to each other, under the appropriate change of coordinates. In most situations, we will use the shorthand $\phi_u = L(u)$. One application that will be used throughout the paper is the following. The last item is often taken as the definition of the Legendre transform (the so called Legendre-Fenchel transform), as it can be used to define $L(\phi)$ for $\phi$ merely continuous. Henceforth we will take $(v)$ as the definition of $L(\phi)$ in any case where $\phi$ is not necessarily $C^1$.

\begin{lemma}[{c.f. \cite[Lemma 2.6]{BB}}] \label{propernesslemma}
	Let $\phi$ be any strictly convex function on an open convex domain $\Omega' \subset \R^n$. Let $u$ be its Legendre transform defined on $\Omega$. If $0 \in \Omega$, then there exists a $C  >0$ such that 
		
	\begin{equation} \label{properness}
		\phi(\xi) \geq C^{-1}|\xi| - C.
	\end{equation}
In particular, $\phi$ is proper.  Moreover, we can estimate $C$ the following way. Let $\varepsilon >0$ be sufficiently small so that $B_\varepsilon(0) \subset \Omega$, then 
	\begin{equation}\label{propernesssharp}
			\phi(\xi) \geq \varepsilon |\xi| - \sup_{B_{\varepsilon}(0)} L(\phi).
	\end{equation}
\end{lemma}

The estimate \eqref{propernesssharp} is an immediate corollary of Lemma \ref{lem2-16}. For smooth functions one can see \eqref{properness} directly; since $0$ is in the domain of $u$, there is some $\xi$ such that $\nabla \phi_u(\xi) = 0$. Then $\phi_u$ is a strictly convex function with a minimum, and hence must grow at least linearly.  However in what follows we will need to make use of \eqref{propernesssharp} even for smooth functions. Although the notation is suggestive of the situation where $\phi$ is the K\"ahler potential of a toric metric, it is worth noting, and will be used later on, that this is completely symmetric in $\phi$ and $u$. That is to say, if $0$ lies in the domain $\Omega'$ of $\phi$, it follows that $u$ must also satisfy \eqref{properness} (with respect to the coordinate $x$ in $\Omega$). Indeed the entirety of Lemma \ref{propernesslemma} is entirely symmetric in $u$ and $\phi$.

   We collect some further elementary properties of the behavior of convex functions under the Legendre transform, all consequences of the properties laid out in Lemma \ref{lem2-16}. As we will see, these in turn give rise to interesting geometric consequences when interpreted in the context of K\"ahler geometry on complex toric manifolds.

\begin{lemma} \label{lem2-18}
	Let $\phi$ be a strictly convex function on $\t$ and $u = L(\phi)$ be its Legendre transform. Let $\Omega$ denote the image of the gradient $\nabla \phi: \t \to \t^*$.
		\begin{enumerate}
			\item For $B \in \text{GL}(n, \Z)$, set $\phi_B(\xi) = \phi \left( B \xi \right)$. Then $L(\phi_B)(x) = u((B^{T})^{-1} x)$, and the image of $\nabla \phi_B: \t \to \t^*$ is equal to $B^T(\Omega)$. 
			\item For $b_1 \in \t$, set $\phi_{b_1}(\xi) = \phi(\xi - b_1)$. Then $L(\phi_{b_1})(x) = u(x) + \langle b_1, x \rangle$. Clearly, the image of $\nabla \phi_{b_1}$ is also equal to $\Omega$.
			\item Symmetrically, for $b_2 \in \t^*$, set $\phi^{b_2}(\xi) = \phi(\xi) + \langle b_2 , \xi \rangle$. Then $L(\phi^{b_2})(x) = u(x - b_2)$ and the image of $\nabla \phi^{b_2}$ is equal to $\Omega - b_2$. 
		\end{enumerate}
\end{lemma} 

Let $M$ be a complex toric manifold together with a K\"ahler metric $\omega$ with respect to which the real $T^{n}$ action is Hamiltonian, and let $\phi$ be a strictly convex function on the dense orbit $\Cstarn \subset M$ such that $\omega = 2 i \p \bp \phi$. Let $\mu: M \to \t^*$ denote the corresponding moment map, normalized so that $\langle \mu, b \rangle = \langle \nabla \phi, b \rangle$ on the dense orbit as in Section 2.2, and suppose that the image of $\mu$ is equal to a Delzant polyhedron $P$. Recall also from Section 2.2 that we fix a basis $X_1, \dots, X_n$ for $\t$. Then the action of $\text{GL}(n, \Z)$ on $\phi$ corresponds simply to changing this basis by an automorphism of $\Cstarn$. This will be useful to simplify calculations later on, since by the Delzant condition we can use this to assume that $P$ locally coincides with a translate of the positive orthant $\R^{n}_+$ near any vertex. The action of $\t$ on $\phi$ given in $(ii)$ corresponds to composing the $\Cstarn$-action on $M$ with an element of the form $e^{-b_1} \in \Cstarn$. Notice that this is is always induced from the global automorphism $e^{-b_1}: M \to M$ of $M$. The $\t^*$-action of $(iii)$ is most naturally viewed as a modification of the moment map $\mu$ by the action of $\t^*$ on itself by translation. 

 Recall from Proposition \ref{prop2-12} that we are interested in the case where $\phi$ is a solution to \eqref{eqn2-5} on $\t$. Since $\phi$ uniquely determines and is uniquely determined by its Legendre transform $u = L(\phi)$, we can once again make use of the properties laid out in Lemma \ref{lem2-16} to rewrite \eqref{eqn2-5} as a real Monge-Amp\`ere equation for the convex function $u$, defined on the interior of the image of the moment map. We assume as above that this image is equal to a Delzant polyhedron $P$.

  \begin{proposition}\label{prop2-19} 
  Suppose that $M$ is a complex toric manifold and $(\omega,X)$ is a shrinking gradient K\"ahler-Ricci soliton on $M$. Suppose that the $T^n$-action is Hamiltonian with respect to $\omega$, so that, by Proposition \ref{prop2-10}, $\omega$ admits K\"ahler potential determined by a strictly convex function $\phi$ on $\t$ satisfying \eqref{eqn2-5}. Let $u = L(\phi)$ be the Legendre transform, which we assume is defined on the Delzant polyhedron $P$. Then $u$ satisfies the real Monge-Amp\`ere equation

 \begin{equation} \label{eqn2-9}
 	2\left( u_i x^i  -  u(x)\right) - \log\det(u_{ij})  = \langle b_X, x \rangle .
 \end{equation}
 
 \end{proposition}

We return now to the canonical metric $\omega_P$ defined in Proposition \ref{prop2-14}. We have just seen that there is a corresponding symplectic potential $u_P$ on $P$. The main result of \cite{Guil} is an explicit formula for $u_P$, only in terms of the data of $P$, in the case that $\overline{P}$ (and therefore $M$) is compact. This has been generalized in \cite[Theorem 5.2]{BGL} to (essentially) arbitrary polyhedra, the Delzant case included. Let $F_i$, $i = 1, \dots, d$ denote the $(n-1)$-dimensional facets of $P$ with inward-pointing normal vector $\nu_i \in \Gamma$, normalized so that $\nu_i$ is the minimal generator of $\sigma_i = \R_+ \cdot \nu_i$ in $\Gamma$. Let $\ell_i(x) = \langle \nu_i, x \rangle$, so that $\overline{P}$ is defined by the system of inequalities $\ell_i(x) \geq - a_i$, $i = 1, \dots, N$, $a_i\in \R$. Then from \cite{BGL} we have the following explicit formula for $u_P$:

\begin{equation} \label{eqn2-10}
	u_P(x) = \frac{1}{2} \sum_{i=1}^d (\ell_i(x) + a_i) \log\left( \ell_i(x) + a_i \right).
\end{equation}

\subsection{Equivalences}

Thus far, we have shown that associated to any Delzant polyhedron $P$ there is a toric K\"ahler manifold $(M_P, J_P, \omega_P)$. We begin this subsection by giving conditions under which we can extend the Delzant classification to the non-compact setting. In brief, we would like to understand the answers to the following questions. First, given a toric K\"ahler manifold $(M,J,\omega)$, under what conditions is the image of the moment map equal to a Delzant polyhedron $P$? Second, given a toric K\"ahler manifold $(M,J,\omega)$ whose moment image is equal to a Delzant polyhedron $P$, under what conditions can we say that $(M,J) \cong (M_P, J_P)$ and $(M, \omega) \cong (M_P, \omega_P)$?

To a large extent these questions have already been studied, and much of what appears below is simply a collection of existing results, rephrased in order to better suit the current setup. The answer to the first question and part of the second comes from the work of \cite{KarLer,PW}. 

\begin{lemma} \label{lem2-20}
	 Let $(M,\omega)$ be any symplectic toric manifold with finite fixed point set. Suppose that there exists $b \in \mathfrak{t}$ such that the function $\langle \mu, b \rangle :M \to \R$ is proper and bounded from below. Then the image of the moment map $\mu$ is a Delzant polyhedron $P$, and moreover $(M,\omega)$ is equivariantly symplectomorphic to $(M_P, \omega_P)$.
\end{lemma}

\begin{proof}
	Since the fixed point set of the $T^{n}$-action is finite, it follows from \cite[Theorem 4.1]{HNP} (c.f. \cite[Proposition 1.4]{PW} and the preceeding remarks) that the existence of such a $b \in \t$ is sufficient to show that the image of the moment map $\mu$ is a polyhedral set in $\t^{*}$. This means by definition that $\mu(M)$ is equal to the intersection of finitely many half spaces. It then follows immediately from \cite[Proposition 1.1]{KarLer} that $P$ is a Delzant (unimodular) polyhedron. Finally, \cite[Theorem 1.3, c.f. Theorem 6.7]{KarLer} furnishes the desired equivariant symplectomorphism.
\end{proof}

Given a general symplectic toric manifold $(M,\omega)$ satisfying the conditions of Lemma \ref{lem2-20}, let $P$ be the corresponding polyhedron in $\t^{*}$. Suppose that there is a compatible complex structure $J$ such that $T^{n}$ acts holomorphically. When $M$ is compact, $J$ is determined up to biholomorphism by $P$. This follows in part since we can always use $J$ to complexify the $T^{n}$-action to an action of the full $\Cstarn$. In general, the issue is more subtle. The following example illustrates the problem.\footnote{We thank Vestislav Apostolov for providing this example.} 

\begin{example}\label{eg2-21}
	Let $(\D, \omega)$ denote the Poincar\'e model of the hyperbolic metric on the unit disc in $\C$. The standard $S^{1}$-action on $\C$ restricts to an action on $\D$, but clearly this does not admit a complexified action of $\Cstar$ on $\D$. The symplectic form $\omega$ is $S^{1}$-invariant and, with an appropriate normalization, the moment map $\mu: \D \to \R$ has image equal to the unbounded closed interval $P = [0, \infty)$. Thus, the image is the Delzant polyhedron $P$, but $\D \not\cong M_P \cong \C$. 
\end{example}

However, if we assume \emph{a priori} that there exists a complexified action, then it does indeed follow that the complex structure must be biholomorphic to the standard one $J_P$ on $M_P$. Let $(M,J)$ be a complex toric manifold, so that there exists an effective holomorphic $\Cstarn$-action. Suppose that $\omega$ is the K\"ahler form of a compatible K\"ahler metric such that the $T^{n}$-action is Hamiltonian. 

\begin{lemma}[c.f. {\cite[Proposition A.1]{Ab2}}]\label{lem2-22}
	Let $(M,J, \omega)$ be as above, and assume that the image of the moment map is equal to a Delzant polyhedron $P$. Then $M$ is equivariantly biholomorphic to $M_P$. In particular, $(M,J)$ is quasiprojective. 
\end{lemma}

\begin{proof} 
	As usual, choose a point $p$ in the interior of the dense orbit $\Cstarn \subset M$, and further choose points $x_i$ in the interior of each $k$-dimensional face $F_i$ of $P$. By \cite[Theorem 4.1, part (v)]{HNP} (c.f. \cite{GS,Prato}), each point $q \in \mu^{-1}(F_i)$ is stabilized by a common torus $T^{n-k}_{F_i} \subset T^{n}$ with Lie algebra $\t_i$, and moreover $F_i$ lies as an open subset of the dual $k$-plane $\t_{F_i}^{\perp} \subset \t^{*}$. By the holomorphic slice theorem \cite[Theorem 1.24]{Sj}, there exists a $\Cstarn$-invariant open neighborhood $U_i \subset M$ of the orbit $\Cstarn \cdot p_i \subset M$ and an equivariant biholomorphism $\Phi_i: U_i \to (\Cstar)^{k} \times \C^{n-k}$, with the standard $\Cstarn$-action such that $\Phi_i(x_i) = (1,0)$ and $\Phi_i(\mu^{-1}(F_i) \cap U_i) = (\Cstar)^{k} \times \{0\}$. We see that the stabilizer $T^{n-k}_{F_i}$ acts in the coordinates induced by $\Phi_i$ by the standard action on $\C^{n-k}$.  Note that the equivariance of $\Phi_i$ ensures that entire dense orbit lies in $U_i$, and hence we can modify the map $\Phi_i$ by the $\Cstarn$-action to ensure that $\Phi_i(p) = (1, \dots, 1)$.  In this way, we produce an equivariant holomorphic coordinate covering of $M$ by running through each $F_i$. Suppose now that $F_1, F_2$ are two $k$-dimensional faces that which lie on the boundary of a higher-dimensional face $E$ of $P$, and let $\Phi_{F_1}: U_{F_1} \to  (\Cstar)^{k} \times \C^{n-k} , \Phi_{F_2}: U_{F_2} \to  (\Cstar)^{k} \times \C^{n-k}, \Phi_E: U_{E} \to  (\Cstar)^{l} \times \C^{n-l}$ denote the corresponding maps as above. By equivariance, the transition map $\Phi_{F_2} \circ \Phi_{F_1}^{-1}$ is uniquely determined by the inclusions of $ (\Cstar)^{l} \times \C^{n-l} \subset (\Cstar)^{k} \times \C^{n-k}$ given by $\Phi_E$, as $E$ varies across all faces containing $F_1$ and $F_2$. These in turn are determined uniquely by the inclusions of the stabilizer algebra $\t_E \subset \t_{F_1}, \t_{F_2}$. As we have seen, the stabilizer algebras $\t_E, \t_{F_1}, \t_{F_2}$ comprise the normal directions to the faces $E, F_1, F_2$ in $\t^*$, respectively. In particular, the transition data of this covering is determined uniquely by the normal fan $\Sigma_P$ of $P$. Now let $(W_i, \Psi_i)$ be a cover of $M_P$ constructed in the same way. For each face $F_i$ of $P$, we have maps $\Psi_i^{-1} \circ \Phi_i: U_i \to W_i$. Since the transition data for each covering is uniquely determined by $\Sigma_P$, we see that these local maps patch together to form a well-defined biholomorphism $M \to M_P$.
\end{proof}

We have thus far met several inequivalent definitions of what it means for a non-compact manifold to be ``toric.'' To avoid confusion, we introduce the following definition, which lies at the intersection of all of the previously introduced notions. 

\begin{definition} \label{def2-23}
	We say that $(M,J,\omega)$, together with a given $\Cstarn$-action is \emph{algebraic-K\"ahler toric} (AK-toric) if the following conditions are met: 
	\begin{enumerate}
		\item The $\Cstarn$-action is effective and holomorphic with respect to $J$.
		\item The symplectic form $\omega$ is the K\"ahler form of a compatible, $T^{n}$-invariant K\"ahler metric on $M$.
		\item The $T^{n}$-action is Hamiltonian with respect to $\omega$, and the moment map $\mu: M \to \t^{*}$ has image equal to a Delzant polyhedron $P$. 
	\end{enumerate}
\end{definition}

Such an $M$ is always equivariantly biholomorphic to the algebraic toric variety $\M_P$ by Lemma \ref{lem2-22} and Proposition \ref{prop2-15}. When $(M,\omega)$ is a compact toric manifold, the polytope $P$ is determined up to translation in $\t^{*}$ by the cohomology class $[\omega]$ \cite{Ab1,Ab2,Guil}. We show this is true in the case that there is an action of the full $\Cstarn$. 
 
\begin{proposition} \label{prop2-24}
	If $(M, J, \omega)$ be AK-toric, then the moment polyhedron $P$ is determined up to translation by the cohomology class $[\omega]$. 
\end{proposition}

\begin{proof}
	The polyhedron $P$ determines a torus-invariant divisor $D_\omega$ on $(M,J)$ as follows. Since $(M,J)$ is biholomorphic to $(M_P, J_P)$, we use this biholomorphism and assume without loss of generality that $(M,J,\omega) = (M_P, J_P, \omega)$ with $\omega$ not necessarily equal to $\omega_P$. Recall that $(M_P, J_P)$ naturally carries the structure of the algebraic toric variety $\M_P$. Thus, we can identify the normal fan $\Sigma$ of $P$ with the fan corresponding to $\M_P$. Let $\nu_i$ be the minimal generator in $\Gamma$ of the ray $\sigma_i \in \Sigma$ corresponding to the direction normal to each facet $F_i$ of $P$. Then each $F_i$ of $P$ has the local defining equation $\ell_i(x) + a_i = 0$, where $\ell_i(x) = \langle \nu_i, x \rangle$ for some $a_i \in \R$. Recall that $\sigma_i$ defines via the Orbit-Cone correspondence an irreducible Weil divisor $D_i$. The divisor $D_\omega$ is then given by 
	
	\begin{equation*} 
		D_\omega = \sum a_i D_i.
	\end{equation*}
We can assume without loss of generality that the irreducible component $D_1$ of $D_{\omega}$ is compact. If there is no such $D_1$, then it follows that there is a $b \in \R^n$ and $A \in \text{GL}(n, \Z)$ such that the affine transformation $Ax + b$ takes $P$ to the positive orthant $\R^n_+$, and so $M \cong \C^n$. Note that the entire construction behaves well with respect to restriction, so that $D_1 = \M_{F_{1}}$. Since $P$ is Delzant, so is $F_1$, and so it follows that $D_1$ is a nonsingular projective variety. If we restrict $\omega$ to $D_1$, we obtain a moment map for the $T^{n-1}$-action $\mu_{1}: D_1 \to \mathfrak{t}_1$, where $\mathfrak{t}_1 \subset \mathfrak{t}$ is the orthogonal complement of the stabilizer algebra of $D_1$. Then the image of $\mu_1$ is the face $F_1$ of $P$ corresponding to $D_1$. After potentially acting by an element of GL$(n, \Z)$, we can assume that $\langle \nu_1, x \rangle = x_1$, so that $\t_1$ can be identified with the subspace $x_1 = 0$. Inside of $\t_1$, $F_1$ is then defined by $\langle \eta_i, (x_2, \dots, x_n) \rangle \geq - \alpha_i$ for some $\eta_i$ in the lattice and $\alpha_i \in \R$. Thus, the Delzant polytope $F_1$ determines a divisor $\Delta = \sum \alpha_i \Delta_i$ on $D_1$, where $\Delta_i$ are the torus-invariant divisors on $D_1$ corresponding to $\eta_i$ through the Orbit-Cone correspondence. 

Since $(D_1, \omega|_{D_1})$ is itself a compact symplectic toric manifold, we can now appeal to the well-established theory in the compact setting \cite{Atiyah,GS,Del,Guil}. Specifically, we have that the cohomology class of the symplectic form $\omega|_{D_1}$ is given by \cite{Del,Guil}
	 
	 \begin{equation*} 
	 	 [\omega|_{D_1}] = \sum \alpha_i [\Delta_i].
	 \end{equation*}  
The coefficients $\alpha_i$, by definition, fix the defining equations of $F_1$ inside $\t_1$. Thus, we see that the facet $F_1$ is uniquely determined by $[\omega]$ up to translation in $\t_1$. By the Orbit-Cone correspondence, the subspace $\t_1$ on which $F_1$ lies is uniquely determined by the fixed fan $\Sigma$, up to translation in its normal direction. We see then that the set of vertices $\{v_1, \dots, v_k\}$ of $F_1$, which is the image under $\mu$ of the set of fixed points $T^n$-action that lie in $\mu^{-1}(F_1)$, is determined uniquely up to a translation in $\t^*$ by $[\omega]$. Now each vertex of $P$ lies on at least one compact facet, again unless $M \cong \C^{n}$ and $P = \R^{n}_+$. Hence, we can repeat this process for each compact torus-invariant divisor to see that the set of \emph{all} vertices $\{v_1, \dots, v_K \}$ of $P$  is determined up to translation in $\t^{*}$ by $[\omega]$. It is clear then that the same is true of $P$. 
\end{proof}

\begin{corollary} \label{cor2-25}
Let $M$ be AK-toric with polyhedron $P = \{ x \in \t^* \: | \: \langle \nu_i, x \rangle \geq - a_i \text{ for all } i = 1, \dots, N \}$, and suppose that $\omega$ is the curvature of an equivariant hermitian holomorphic line bundle $(L, h)$. Then $L \cong \mathcal{O}(D_\omega)$ is the line bundle associated to the divisor $D_\omega  = \sum a_i D_i$.
\end{corollary}

\begin{proof}
	Recall that an AK-toric manifold with polyhedron $P$ is biholomorphic to the toric variety $\M_P$. Let $\Sigma$ be the normal fan of $P$ so that $\M_P = \M_\Sigma$. Since $M$ is smooth we have by \cite[Proposition 4.2.6]{CLS} that $L \cong \mathcal{O}(D)$ for some torus-invariant divisor $D = \sum \beta_i D_i$ with $\beta_i \in \Z$. We let $P_D$ denote the polyhedron associated to $D$ given by \eqref{polyofdivisor}, i.e. 
	\begin{equation*}
		P_D = \left\{ x \in \t^{*} \: | \: \langle x , \nu_i \rangle \geq - \beta_i, \text{ for all } i = 1, \dots, N \right\},
	\end{equation*}
where $\nu_1, \dots, \nu_N$ are the minimal generators of the rays $\sigma_i \in \Sigma$. If $D$ and $D'$ are any two torus-invariant divisors on $M$ with integer coefficients, we define an equivalence relation by declaring that $D \sim D'$ if and only if there exists some $\nu \in \Gamma^*$ such that $P_{D'} = P_D + \nu$, where $P_D$ and $P_{D'}$ are the polyhedra defined in \eqref{polyofdivisor}. By \cite[Theorem 4.1.3]{CLS}, $D \sim D'$ if and only if $\mathcal{O}(D) \cong \mathcal{O}(D')$. Suppose that $D_1$ is a compact torus-invariant Weil divisor in $M$. As before, such a $D_1$ must exist unless $M \cong \C^{n}$ and $P = \R^{n}_+$. Perhaps by modifying $D$ by the equivalence relation, we can assume that the coefficient $\beta_1$ corresponding to $D_1$ is zero. In other words, there is a section $s_1$ of $L$ which does not vanish identically on $D_1$. Let $F_1 \subset \overline{P}$ be the facet corresponding to $D_1$. As before, the Delzant polyhedron $F_1$ determines a unique torus-invariant Weil divisor $\Delta = \sum \alpha_i \Delta_i$ on $D_1$. The restriction of $s_1$ to $D_1$ is a section of $L|_{D_1}$ which vanishes along $\Delta_i = D_1 \cap D_i$ to order $\alpha_i$. In particular, we see that the coefficients $\alpha_i$ of $\Delta_i$ are equal to those $\beta_i$ such that $D_i \cap D_1 \neq \emptyset$. Recall that $D_\omega = \sum a_i D_i$. We claim that $D_\omega \sim D$. As before, we can act by GL$(n, \Z)$ so that $\langle \nu_1, x \rangle = x_1$. Write $P_1 = P_\omega + \nu_1$ so that the face $F_1 + \nu_1$ corresponding to $D_1$ now lies on the hyperplane $x_1 = 0$, and in general $P_1$ is defined by $\langle x, \nu_i \rangle \geq \langle \nu_1, \nu_i \rangle - a_i = -\tilde{a}_i$. Then it is straightforward to compute that the coefficients $\alpha_i$ are equal to those $\tilde{a}_i$ such that $D_i \cap D_1 \neq \emptyset$. Running across all compact divisors of $M$, we see that the coefficients $a_i$ in the defining equations for $P_\omega$ are uniquely determined by $\beta_i$ up to equivalence. In particular, $D_\omega \sim D$. 
\end{proof}

\begin{proposition}\label{prop2-26}
	Let $(M,J,\omega)$ be AK-toric with moment polyhedron $P$. Then  $\omega$ admits a strictly convex symplectic potential $u$ on $P$, unique up to the addition of an affine function on $P$. Moreover, the function $u$ takes a special form. Recall that $M \cong M_P$, so that $P$ in particular determines a K\"ahler form $\omega_P$ on $M$ with symplectic potential $u_P$ defined by \eqref{eqn2-10}. Then there exists a function $v \in C^\infty(\overline{P})$ such that 
	
	\begin{equation} \label{eqn2-11}
		u = u_P + v.
	\end{equation}
	
\end{proposition}

\begin{proof}
	By Proposition \ref{prop2-10}, the restriction of $\omega$ to the dense orbit $\Cstarn \subset M$ is determined by a strictly convex function $\phi$ on $\t$. The moment map $\mu: M \to \t^{*}$ is then determined by the Euclidean gradient $\nabla \phi$ on $\t$. Thus, there is a symplectic potential $u = L(\phi)$ defined by the Legendre transform \eqref{eqn2-6}. That $u$ satisfies the boundary condition \eqref{eqn2-11} follows from \cite[Proposition 1]{ACGT}. Indeed, the key point is that the Hessian $(u_{ij})$ of $u$ determines a natural complex structure $J_u$ on the dense open subset $\mu^{-1}(P) \subset M$.  In the compact setting, it was proved by Abreu \cite{Ab3} using the global symplectic slice theorem that the boundary conditions \eqref{eqn2-11} are equivalent to the fact that the complex structure $J_u$ extends to all of $M$.  Passing via the Legendre transform to complex coordinates on $\Cstarn \subset M$, we see that the extension of $J_u$ to all of $M$ is then equivalent to the extension of the symplectic form $\omega = 2i\p\bp \phi$ to $M$. 
	
	These arguments were then improved by Apostolov-Calderbank-Gauduchon-T{\o}nnesen-Friedman \cite{ACGT} and independently by Donaldson \cite{Donaldson-Abreu} who derived the boundary conditions \eqref{eqn2-11} from purely local considerations.  Indeed, the proof of \cite[Proposition 1]{ACGT} proceeds by showing that if $F$ is any $k$-dimensional face of $P$, then for each point $y \in F$, the Hessian $(u - u_P)_{ij}$ extends smoothly in a neighborhood of $y$.  This is achieved by choosing an arbitrary point $q \in \mu^{-1}(y)$ and a local symplectic slice (for the $T^n$-action), and then using the Taylor expansion for the metric around the point $q$ to prove that the complex structure $J_u$ defined by $u$ extending smoothly to $q$ is equivalent to the boundary condition \eqref{eqn2-11} at $y \in F$, which applies verbatim in our setting. Translating back to the complex picture via the Legendre transform we see that this in turn is equivalent to the condition that the K\"ahler metric $\omega = 2i\p\bp \phi$ extends to the subvariety $V_F$ corresponding to the face $F$ given by the Orbit-Cone correspondence (Proposition \ref{prop2-7}),  using the fact that $V_F$ is naturally identified with $\mu^{-1}(F)$ by Lemma \ref{lem2-22}.  As in the compact case, since the boundary of $P$ is piecewise linear and since $u-u_P$ is smooth on the interior, this implies that $u-u_P$ itself extends smoothly to a neighborhood of each point $y \in F$.  This completes the proof noting that $F$ and $y$ are arbitrary. 
\end{proof}

\begin{remark}\label{proper-moment-map-business}
It should be noted, although it is not needed for our purposes, that this also holds under somewhat more general conditions.  In particular let $(M,\omega)$ be a $2n$-dimensional symplectic toric manifold together with a compatible complex structure $J$, making $(M,J, \omega)$ into a K\"ahler manifold (recall that this means that $(M,\omega)$ admits a Hamiltonian $T^n$-action, but not necessarily a corresponding $\Cstarn$-action).  Suppose that the moment map $\mu: M \to \t^*$ is proper, and as usual denote by $\overline{P}$ the image $\mu(M) \subset \t^*$. Then one can still define a symplectic potential $u$ by considering the complex structure $J_u$ associated to $u$ on $P$, but it is no longer evident a priori in this setting that the metric $g$ associated to $(M, J)$ can be written in the form \eqref{eqn2-7} since we do not have a corresponding K\"ahler potential $\phi$ furnished by Proposition \ref{prop2-10}.  However, using the properness of $\mu$, one can apply Lemma \ref{lem2-20} to show that there is a globally defined isometry between $g$ and a metric $g'$ which is defined on the interior of $P$ by \eqref{eqn2-7}, and then correspondingly deduce the boundary conditions \eqref{eqn2-11} from \cite[Proposition 1]{ACGT} as above.  This was the approach of the recent work of Sena-Dias in \cite[Section 3]{SD} to prove a uniqueness result for scalar-flat metrics on non-compact toric 4-manifolds which are not necessarily complex toric.
\end{remark}

\section{Convexity properties}

\subsection{The weighted volume functional}

Let $(N, \omega)$ be a Fano manifold with a given K\"ahler metric $\omega \in 2\pi c_1(N)$, and let $\mathfrak{h}$ be the space of all holomorphic vector fields on $N$. Given $v \in \mathfrak{h}$, let $\theta_v$ be a Hamiltonian potential for $Jv$ with respect to the $T^k$-action generated by the flow of $Jv$, which exists because in the compact manifolds with $c_1 > 0$ always satisfy $H^1(N)=0$. Then set $F(v)$ as 

\begin{equation*} 
	F(v) = \int_N e^{-\theta_v} \omega^n.
\end{equation*}
In order for this to be well-defined of course one must normalize $\theta_v$. With an appropriate choice, it turns out that $F(v)$ is independent of choice of the metric $\omega$ in its cohomology class \cite{TZ2}. The modified Futaki invariant of \cite{TZ2} is then defined as the derivative $F_X:\mathfrak{h} \to \C$ of $F$ at a given holomorphic vector field $X$. Then $F_X$ is independent of the choice of reference metric, and in \cite{TZ2} it is shown that $F_X$ must therefore vanish identically if $X$ is the vector field corresponding to a K\"ahler-Ricci soliton on $N$. A necessary condition therefore for $X$ to occur as the vector field of a shrinking gradient K\"ahler-Ricci soliton on $N$ is that $F_X \equiv 0$.

 It is shown in \cite{ConDerSun} that  these ideas can be generalized to the non-compact setting in the presence of a complete shrinking gradient K\"ahler-Ricci soliton with bounded Ricci curvature. As in \cite{ConDerSun}, we refer to $F$ as the \emph{weighted volume functional}. Suppose that a real torus $T^{k}$ acts on $M$ holomorphically and effectively with Lie algebra $\t$, and that the soliton vector field $X$ satisfies $JX \in \t$. By the Duistermaat-Heckman theorem \cite{DH,DHadd,PW}, there is an open cone $\Lambda \subset \t \subset \mathfrak{h}$ where the weighted volume functional $F$, and thereby the Futaki invariant, can be defined. Moreover, the domain $\Lambda$ can be naturally identified with the dual asymptotic cone of $\mu(M) \subset \t^*$ (see \cite[Definition A.2, Definition A.6]{PW}). Just as in \cite{PW}, we will see that $\Lambda$ is in natural bijection with the space of Hamiltonian potentials which are proper and bounded below on $M$. In this setting, the soliton vector field $X$ has the property that $JX \in \Lambda$ and is the unique critical point of $F$ \cite[Lemma 5.17]{ConDerSun}. This is analogous to the volume minimization principle of \cite{MSY2} for the Reeb vector field of a Sasaki-Einstein metric.
 
 We show that on an AK-toric manifold $M$ with moment polyhedron $P$, the weighted volume functional $F$ is proper, convex, and bounded from below. It is clear from the definitions that the asymptotic cone of $P$ is equal to its recession cone $C$. Thus, there is a natural identification of the domain $\Lambda$ of $F$ with the dual recession cone $C^* \subset \t$. Fix a Delzant polyhedron $P$ and let $M \cong M_P$. Throughout this section we make the extra assumption that $P$ contains the point zero in its interior. This of course can always be achieved by a translation, which corresponds to a modification of the moment map by a constant; see Lemma \ref{lem2-18}. Suppose that there exists an AK-toric metric $\omega$ on $M$ with $P$ as its moment polyhedron. Then there is a potential $\phi$ for $\omega$ on the dense orbit. For any $v \in \t$, we know from Lemma \ref{lem2-11} that there is a fixed $b_v \in \R^n$ such that the restriction of the Hamiltonian potential $\theta_v$ to the dense orbit is determined by the function $\langle b_v, \nabla \phi \rangle$ on $\R^n$. Then passing to symplectic coordinates via the Legendre transform \eqref{eqn2-6}, we then see that $\theta_v$ is determined by the linear function $\langle b_v, x \rangle$ on $P$. The next proposition can be interpreted as the existence and uniqueness of a vector field in $\t$ with vanishing Futaki invariant.

\begin{proposition} \label{prop3-1}
	Let $P \subset \t^*$ be a Delzant polyhedron containing zero in its interior. Then there exists a unique linear function $\ell_P(x)$ determined by $P$ such that 
	
	\begin{equation}\label{eqn3-1}
		\int_P \ell(x) e^{-\ell_P(x)}dx = 0
	\end{equation}
for any linear function $\ell$ on $P$. 
\end{proposition}

\begin{proof}
Of course here $\t^*$ can be any real vector space, although our only application is when $\t^*$ is the dual Lie algebra of a real torus $T^{n}$. Let $C \subset \t^*$ be the recession cone of $P$. It follows immediately from the definition that the interior of $C^*$ is characterized by those $b \in \mathfrak{t}$ such that the linear function $\langle b, x \rangle$ on $P$ is positive outside of a compact set. Indeed, for each $b \in \t$, set 

	 \begin{equation*}
	 	H_{b} = \{ x \in \R^n \: | \: \langle b, x \rangle \leq 0 \},
	\end{equation*}
	and

	 \begin{equation*}
	 	Q_{b} = H_{b} \cap \overline{P}.
	\end{equation*}
We see from the definition (Definition \ref{defi2-3}) that  an element $y \in \t^*$ lies in $C$ if and only if $x + \lambda y \in P$ for all $x \in P$, $\lambda \geq 0$. Thus $Q_b$ is compact if and only if for each $x \in Q_b$,  and for each $y \in C$, there exists a $\lambda > 0$ such that $\langle x + \lambda y, b \rangle = 0$.  Since $\langle x, b \rangle \leq 0$, it follows that $Q_b$ is compact if and only if $b \in C^*$. Thus $ e^{-\langle b , x \rangle}$ is integrable on $P$, and so there is a well-defined function $F: C^* \to \R$ given by 

\begin{equation*} 
	F(b) = \int_P e^{-\langle b , x \rangle} dx.
\end{equation*}
Then 

\begin{equation*} 
	\frac{\p}{\p b^j} F = - \left(\int_P x^j e^{-\langle b , x \rangle} dx\right).
\end{equation*}
Moreover, the critical points of $F$ are precisely solutions $\ell_P$ to \eqref{eqn3-1}. The function $F$ is convex which immediately gives uniqueness. To show existence, it suffices to show that $F$ is proper. That is, given a sequence $b_j$ in the interior of $C^*$ such that either $|b_j| \to \infty$ or the sequence $\{b_j\}$ approaches a point on the boundary, we need to show that $F(b_j) \to \infty$. Consider the former case first. Using the natural inner product on $\t$, we can view the dual recession cone $C^{*}$ as sitting inside of $\t^*$. Since $0 \in P$, the intersection $Q  = - C^{*} \cap P$ has positive measure in $\R^{n}$. Now suppose that $\{b_j\}$ is any sequence in $C^{*}$ such that $|b_j| \to \infty$. Let $y \in Q$ be a fixed point in the interior and choose $\varepsilon$ sufficiently small so that $B_\varepsilon(y) \subset Q$ has strictly positive Euclidean distance to the boundary $\partial Q$. In particular, we then have that $ \inf_{v \in S^{n-1} \cap \overline{C}^*} \langle v, -y \rangle > 0$. We choose $\varepsilon$ sufficiently small so that $\delta =  \inf_{v \in S^{n-1} \cap \overline{C}^*} \langle v, -y \rangle  - \varepsilon > 0 $. For any $x \in B_{\varepsilon}(y)$, write $x = y + r w$ for $r \in [0, \varepsilon)$ and $w \in S^{n-1}$. Then we have, for any $(b, x) \in C^* \times B_{\varepsilon}(y)$, 

\begin{equation*}
	 -\langle b, x \rangle  \geq \langle b , -y  \rangle - r|b| |v| \geq \left( \left\langle \frac{b}{|b|}, -y \right \rangle - \varepsilon \right) |b| \geq \delta |b|.
\end{equation*}
Therefore, we see immediately that

\begin{equation*}
	F(b_j) = \int_P e^{-\langle b_j , x \rangle} dx \geq \int_{ B_{\varepsilon}(y)} e^{-\langle b_j , x \rangle} dx  \geq \int_{ B_{\varepsilon}(y)} e^{ \delta |b_j|} dx.
\end{equation*}
Since $|b_j| \to \infty$, we have then that $F(b_j) \to \infty$. 

Consider now the latter case. The key point is that $\partial C^*$ is defined by those $\bar{b} \in \R^n$ such that there exists at least one $\bar{c} \in \overline{C}$ with $\langle \bar{b}, \bar{c} \rangle = 0$. Choose $\bar{b} \in \p C^*$. The result essentially follows from the fact that the polyhedron $Q_{\bar{b}}$ defined above is unbounded. More explicitly, if $\bar{c}$ is a point with $\langle \bar{b}, \bar{c} \rangle = 0$, then for any $x_0 \in Q_{\bar{b}}$ we have that $x_0 + \lambda c \in Q_{\bar{b}} $ for any $\lambda \geq 0$. If we then fix a small $(n-1)$-disc $D_\varepsilon(x_0) \subset Q_{\bar{b}}$ perpendicular to $c$, consider the tubes $T_\lambda = \{ x + r c \: | \: x \in D_\varepsilon(x_0), r \in (0, \lambda) \} \subset Q_{\bar{b}} $. Take a sequence of points $b_i \to \bar{b}$ with $b_j$ in the interior of $C^*$, and define $Q_{b_j}$ and $H_{b_j}$ as above. Recall that each $Q_{b_j}$ is bounded. Choosing $\varepsilon$ small enough, and perhaps after removing finitely many terms from $\{b_j\}$, we can assume that $D_{\varepsilon} (x_0)$ is contained in $Q_{b_1}$. Let $\lambda_j$ be the largest positive number such that $T_{\lambda_j} \subset Q_{b_j}$. Since $Q_{b_j} \to Q_{\bar{b}}$, we see that $\lambda_j \to \infty$. Then we have 

\begin{equation*}
	F(b_j) = \int_P e^{-\langle b_j , x \rangle} dx \geq \int_{T_{\lambda_j}}e^{-\langle b_j , x \rangle} dx = \lambda_j \int_{D_\varepsilon(x_0)} e^{-\langle b_j, y \rangle } dy,
\end{equation*}
where $y$ are the coordinates on $D_\varepsilon(x_0)$. Clearly $F(b_j) \to \infty$.
\end{proof}

\begin{corollary} \label{cor3-2}
	Let $P \subset \R^{n}$ be a Delzant polyhedron, $M = M_P$, and suppose that $\omega$ is a $T^n$-invariant K\"ahler metric with $P$ as its moment polyhedron. Let $v$ be the holomorphic vector field on $M$ determined by $b_v \in \mathfrak{t}$ and $\theta_v$ be a Hamiltonian potential for $Jv$. Then 
	
	\begin{equation*} 
		\int_M e^{-\theta_v} \omega^n < \infty
	\end{equation*}
if and only if $b_v$ lies in the dual recession cone $C^*$. 
	
\end{corollary}

\begin{proof}
	We work on the dense orbit in symplectic coordinates $\Cstarn \cong P \times T^n$. We have seen in Section 2.4 that in these coordinates $\omega$ is given simply by $\omega = \sum dx^i \wedge d\theta^i$ so that the integral above becomes 
	
	\begin{equation*} 
		 \int_{\Cstarn}e^{-\theta_v} \omega^n = \int_{P\times T^n} e^{-\langle b_v, x \rangle} dx d\theta = (2\pi)^n \int_{P} e^{-\langle b_v, x \rangle} dx.
	\end{equation*}
As we have seen, this is finite precisely when $b_v \in C^*$. 
\end{proof}

As a consequence, we recover the result of \cite{PW} that domain the $\Lambda$ of the weighted volume functional $F$ can be identified with the dual asymptotic cone $C^*$.

\subsection{The soliton equation}

 Let $P$ be a Delzant polyhedron containing zero in its interior and $M \cong M_P$. Suppose that there is a complete $T^{n}$-invariant shrinking gradient K\"ahler-Ricci soliton $\omega$ with $P$ as its moment polyhedron, and whose soliton vector field $X$ satisfies $JX \in \t$. From Proposition \ref{prop2-19}, we know that there is a corresponding symplectic potential $u \in C^\infty(P)$ which satisfies
 
 \begin{equation*}
 	2\left( u_i x^i  -  u(x) \right) - \log\det(u_{ij})  = \langle b_X, x \rangle,
 \end{equation*}
where the linear function $\langle b_X, x \rangle$ on $P$ corresponds via the Legendre transform to the Hamiltonian potential $\theta_X = \mu(JX)$ for $JX$. We adopt the following simplification of notation from \cite{Don1}. For a given $u \in C^\infty(P)$, set
 
 \begin{equation} \label{rhodef}
	\rho_u =  2\left( u_i x^i  -  u(x) \right) - \log\det(u_{ij})  
\end{equation}
so that the soliton equation can once again be rewritten as

\begin{equation} \label{rhoequalsbsol}
	\rho_u = \langle b_X, x \rangle .
\end{equation}
The function $e^{-\rho}$ is natural to study in the context of integration over $P$. In particular,

\begin{corollary} \label{cor3-3}
	Let $P$ be a Delzant polyhedron containing zero in its interior. For any smooth and convex function $u$ on $P$, we have that 
	
	\begin{equation*}
		\int_P e^{-\rho_u }dx < \infty.
	\end{equation*}
	
\end{corollary}

\begin{proof} 
	 To prove the corollary, we let $\phi_u(\xi) = L(u)$ be the Legendre transform and apply the change of coordinates $x = \nabla \phi_u(\xi)$, where $\xi$ denotes coordinates on the domain $\Omega \subset \R^n$ of $\phi_u$. Then from Lemma \ref{lem2-16} we have 
	
	\begin{equation*}
		\det(u_{ij}) dx = d \xi,
	\end{equation*}
and 
	
	\begin{equation*}
		u - \langle \nabla u, x \rangle = - \phi_u(\xi).
	\end{equation*}
Therefore 

	\begin{equation*}
		\int_P e^{-\rho_u} dx = \int_{\Omega } e^{-2\phi_u} d\xi.
	\end{equation*}
Then from Lemma \ref{propernesslemma} we know that $e^{-2\phi_u}$ is integrable on $\Omega$. 
\end{proof}

\begin{remark} \label{rmk3-4}
	We emphasize at this stage the statement of Lemma \ref{propernesslemma}; simply by asserting that zero lies in the \emph{domain} of $u$, it follows automatically that the Legendre transform $\phi_u$ of $u$ is proper.
\end{remark}

\begin{corollary} \label{cor3-5}
	Let $P$ be a Delzant polyhedron containing zero in its interior, and suppose that there exists a solution $u \in C^\infty(P)$ to \eqref{rhoequalsbsol}. Then the element $b_X \in \t$ determining $JX$ lies in $C^*$.
\end{corollary}

\begin{proof}
	 Since $P$ contains zero in its interior, we have by Corollary \ref{cor3-3} that 
	
	\begin{equation*}
		\int_P e^{-\rho_u} dx < \infty.
	\end{equation*}
Since $u$ satisfies \eqref{rhoequalsbsol}, we have 

	\begin{equation*}
		\int_P e^{-\langle b_X, x \rangle}dx < \infty.
	\end{equation*}
Since the restriction of the Hamiltonian potential $\theta_X$ for $JX$ to $P \times T^n$ is given by $\theta_X|_{P \times T^n} = \langle b_X, x\rangle$, it follows from Corollary \ref{cor3-2} that $b_X \in C^*$. 
\end{proof}

\begin{lemma} \label{lem3-6}
	Let $P$ be a Delzant polyhedron containing zero in its interior, and suppose that there exists a solution $u \in C^\infty(P)$ to \eqref{rhoequalsbsol}. Then the linear function $\langle b_X, x \rangle$ on $P$ satisfies 
	
	\begin{equation*}
		\int_P \ell(x) e^{-\langle b_X, x \rangle}dx = 0
	\end{equation*}
for any linear function $\ell(x)$ on $P$.
	
\end{lemma}

\begin{proof}
First, we claim that any function $u \in C^\infty(P)$ which is the Legendre transform of a smooth convex function $\phi$ on $\R^n$ satisfies 
	
	\begin{equation*}
		\int_P \ell(x) e^{-\rho_{u}}dx = 0 
	\end{equation*}
for any linear function $\ell(x)$ on $P$. Pick any coordinate $x^j$ and compute 

	\begin{equation*} 
		\int_P x^j e^{-\rho_{u}}dx = \int_{\R^n} \phi_j e^{-2\phi}d\xi = - \frac{1}{2} \int_{\R^n} \left( e^{-2\phi}\right)_j d\xi.
	\end{equation*}
By Lemma \ref{propernesslemma}, we know that $e^{-\phi}$ decays at least exponentially in $|x|$. Thus, integration by parts yields that the term on the right-hand side is zero. Then if $u$ satisfies $\rho_{u} = \langle b_X, x \rangle$, it follows that

	\begin{align*}
		 \int_P x^j e^{-\langle b_X, x \rangle } dx = - \frac{1}{2} \int_{\R^n} \left( e^{-2\phi}\right)_j d\xi = 0
	\end{align*}
for each $j$. 
\end{proof}

Therefore, the linear function $\langle b_X, x \rangle$ on $P$ must be equal to the unique linear function $\ell_P$ determined by Proposition \ref{prop3-1}. We will henceforth denote 

\begin{equation*}
	\langle b_X, x \rangle = \ell_P(x)
\end{equation*}
since whenever both sides exist, they must coincide.

\subsection{Real Monge-Amp\`ere equations on unbounded convex domains}

In this section we study the analytic properties of some real Monge-Amp\`ere equations of the same form as \eqref{rhoequalsbsol}. More precisely, we will consider equations of the form

\begin{equation} \label{rhoequalsA}
	\rho_u = A,
\end{equation}
where now the right-hand side $A(x) \in C^\infty(\overline{P})$ can be \emph{any} smooth function satisfying some fixed hypotheses which we will discuss below. When $P$ is bounded, this is also the approach taken in \cite{BB} and \cite{Don1}. Let $P$ be a Delzant polyhedron defined by the system of inequalities $\ell_i(x) + a_i \geq 0$, and suppose that $P$ contains zero in its interior.  Define $u_P$ as in \eqref{eqn2-10} by 

	\begin{equation*}
		u_P(x) =\frac{1}{2} \sum_{i=1}^d (\ell_i(x) + a_i) \log\left( \ell_i(x) + a_i \right),
	\end{equation*}
recalling that $u_P$ is the symplectic potential of the canonical K\"ahler metric $\omega_P$ on $M_P$. Let $A(x) \in C^\infty(\overline{P})$. We will say that $A$ is \emph{admissible} if each of the following conditions hold:

 \begin{enumerate}
 	\item  $V_A = \int_P e^{-A}dx < \infty$,
	\item $\int_P \ell(x) e^{-A}dx = 0 \text{ for any linear function } \ell$,
	\item $\int_P u_{P} e^{-A} dx < \infty$.
 \end{enumerate}
For an admissible function $A$. In analogy with Proposition \ref{prop2-26}, we set $\mathcal{E}_A^{1,\infty}$ to be the set
\begin{equation*}
 \mathcal{E}_A^{1,\infty} = \left\{ u = u_P + v  \, \bigg|\,  \int_P u e^{-A(x)}dx < \infty \, , \, (u)_{ij} > 0\,,\, v \in C^\infty(\overline{P}) \right\},
\end{equation*}
and similarly 
\begin{equation*}
 \mathcal{E}_A^{1,0} = \left\{ u = u_P + v  \, \bigg|\,  \int_P |u| e^{-A(x)}dx < \infty \, , \,  u \textnormal{ is convex} \,,\, v \in C^0(\overline{P}) \right\}.
\end{equation*}
The space $\Pot$ of \emph{symplectic potentials} is then
\begin{equation*} 
	\Pot = \left\{ u \in \mathcal{E}_A^{1,\infty}  \, \bigg|\,  \nabla u: P \to \R^n \textnormal{ is surjective} \right\}.
\end{equation*}
In fact we have $ \Pot \subset \mathcal{E}_A^{1,\infty}  \subset \mathcal{E}_A^{1,0} $.  The first inclusion is clear, and to see the second we proceed as follows.  If $ u \in \mathcal{E}_A^{1,\infty} $, we can modify by a linear function to ensure that $\nabla u(0) = 0$, and since $A$ is admissible this does not affect the value of $\int_P u e^{-A} dx$.  By Lemma \ref{propernesslemma} we can add a constant to $u$ to ensure that $u \geq 0$, and again the admissibility of $A$ ensures that this only affects the value of $\int_P u e^{-A} dx$ by a the addition of a constant. Hence we see that $\int_P |u| e^{-A}dx < \infty$ for any $u \in \mathcal{E}_A^{1,\infty} $.  The space $\Pot$ can be naturally viewed as a convex subset $C^\infty(\overline{P})$.

\begin{lemma}\label{surjective-gradient}
	Suppose that $u_0, u_1 \in \Pot$ and set $u_t = tu_1 + (1-t)u_0$. Then $\nabla u_t:P \to \R^n$ is surjective for all $t \in [0,1]$. 
\end{lemma}

\begin{proof}
We first observe that this is true when $n = 1$.  Indeed, in this case $P$ can be taken to be an interval $(a, b)$ for $a < 0$ and $b \in (0, \infty]$.  Then by convexity $\frac{\p u_t}{\p x}: (a, b) \to \R$ will be surjective if and only if $\lim_{x \to b} \frac{\p u_t}{\p x} = \infty$ and $\lim_{x \to a} \frac{\p u_t}{\p x} = -\infty$. This is a property that $u_t$ clearly inherits from $u_0$ and $u_1$. 

 In general, we suppose for the sake of contradiction that there is some time $t$ such that $\nabla u_t(P) = \Omega \subsetneq \R^n$. Choose $\xi^* \in \p\Omega$ and a sequence $x_i \in P$ such that $\nabla u_t(x_i) \to \xi^*$.  By potentially adding a linear function, we assume without loss of generality that $\nabla u_0(0) = \nabla u_1(0) = 0$.  By passing to a subsequence then we can assume that either $x_i$ accumulate in $\partial P$, $|x_i| \to \infty$, or indeed both.  In either case, it follows from the choice of normalization together with the one-dimensional case that the radial derivative $\frac{\p u_t}{\p r}$ satisfies $\left| \frac{\p u_t}{\p r}(x_i)\right| \to \infty $, and hence $|\nabla u_t|(x_i) \to \infty$. This is a contradiction with the assumption that $\nabla u_t(x_i) \to \xi^*$. 
\end{proof}

Lastly we define $\Pot_0 \subset \Pot$ to be the space of \emph{normalized} symplectic potentials; these will be those $u \in \Pot$ such that 

\begin{equation} \label{eqn3-5}
	\int_P u e^{-A(x)}dx = 0.
\end{equation}
Clearly for any $u \in \Pot$, we can find a constant $c$ such that $u + c \in \Pot_0$.

 \begin{definition} \label{defi3-7}
 	Given any $u_0, u_1\in \Pot$, we say that the linear path $u_t = (1-t)u_0 + t u_1$ joining $u_0$ and $u_1$ is a \emph{geodesic}. 
 \end{definition}
 
We will see that, as a consequence of an elementary local argument,  geodesics in this sense have the property that their Legendre transforms define geodesics in the space of K\"ahler metrics on $M \cong M_P$ in the usual sense. The interpretation is that if $u_0, u_1 \in C^\infty(P)$ are the Legendre transforms of two K\"ahler potentials $\phi_0, \phi_1$ on $M$, then the path $\phi_t = L(u_t)$ solves the pointwise equation 
 
 \begin{equation}\label{geodesic-kahler}
 	\ddot\phi_t - \frac{1}{2} \left | \nabla_{\omega_t}\dot \phi_t \right|^2_{\omega_{t}} = 0 ,
 \end{equation}
and can thus be considered a geodesic in the space of K\"ahler metrics in the sense of \cite{DonSym}. This is a simple exercise in the basic properties of the Legendre transform. We will only make use of a small piece of the computation, but for completeness we include the proof below.
 
 \begin{lemma} \label{lem3-8}
 	Let $u_t$ be any path in $\Pot$ and $\phi_t = L(u_t)$. Then the time derivatives satisfy
	
	\begin{equation}\label{eqn3-6}
		\dot u_t = - \dot \phi_t.
	\end{equation}
Consequently,  if $\ddot{u}_t = 0$ then $\phi_t$ satisfies \eqref{geodesic-kahler}. 
 \end{lemma}
 
 \begin{proof}
 We have
 	\begin{align*}
		\dot u_t = \ddt u_t(x) & = \ddt \left(\big\langle \nabla u_t, x \big\rangle - \phi_t(\nabla u_t) \right) \\
				& =  \bigg\langle \ddt \nabla u_t, x \bigg\rangle - \dot \phi(\nabla u_t) - \bigg\langle \nabla \phi_t, \ddt \nabla u_t \bigg\rangle \\
				& = - \dot \phi ,
	\end{align*}
	which is the first statement. For the second,  note that it follows from Lemma \ref{lem2-16} that 

\begin{align*} 
	 \frac{\partial \dot\phi_t}{\p \xi^j} = - u_t^{ij} \frac{\p \dot u_t}{\p x^i}.
 \end{align*}
 Now compute

\begin{align*}
	\ddot u_t  &= - \ddt \dot \phi_t(\nabla u_t) = - \ddot \phi_t(\nabla u_t) - \sum_m \frac{\p \dot \phi_t }{\p \xi^m}\frac{\p \dot u_t }{\p x^m} \\
			&=  - \ddot \phi_t(\nabla u_t)  + u_t^{lm} \frac{\p \dot u_t }{\p x^l}\frac{\p \dot u_t }{\p x^m},
\end{align*}
so that 

\begin{align*}
	\ddot \phi_t - \frac{1}{2}| \nabla_{\omega_t} \dot \phi_t |_{t}^2 &= \ddot \phi_t - \phi_t^{ij} \frac{ \p \dot \phi_t }{\p \xi^i } \frac{ \p \dot \phi_t }{\p \xi^j } \\
	         &= - \ddot u_t + u_t^{lm} \frac{\p \dot u_t }{\p x^l}\frac{\p \dot u_t }{\p x^m} - (u_t)_{ij}u_t^{il}u_t^{mj}\frac{\p \dot u_t }{\p x^l}\frac{\p \dot u_t }{\p x^m} \\
	         &=  - \ddot u_t  = 0 .
\end{align*}
 \end{proof}
 
 \begin{remark}\label{remark-variation-measurezero}
 While the proof of \eqref{geodesic-kahler} requires two spacial derivatives of $u$ and $\phi$, the proof of the simpler equality \eqref{eqn3-6} works at any point where $u$ and $\phi$ are $C^1$. Since any convex function is $C^1$ outside of a set of measure zero, it follows that \eqref{eqn3-6} actually holds almost everywhere (in the sense that $\dot{u}_t(x) = - \dot{\phi}_t(\nabla u_t(x))$) for any $u \in \mathcal{E}^{1,0}_A$, a fact that we will make use of later on. 
 \end{remark}
 
We introduce a Ding-type functional $\F$ defined on $\mathcal{E}^{1,0}_A$ whose critical points, at least formally, are solutions to \eqref{rhoequalsA}. Define $\F_1$ on $\mathcal{E}^{1,0}_A$ by setting

\begin{equation} \label{eqn3-7}
	\F_1(u) = \int_{\R^n} e^{-2\phi_u}d\xi,
\end{equation}
where $\phi_u = L(u)$. This is well-defined on $\mathcal{E}^{1,0}_A$ by Lemma \ref{propernesslemma}, since the domain $P$ of $u$ contains zero by assumption. In particular, we can extend $e^{-2\phi_u}$ continuously by zero outside of the domain of $\phi_u$ to make sense of the integral \eqref{eqn3-7} over all of $\R^n$.
\begin{remark}\label{ding1-2-remark}
	Whenever $u \in  \mathcal{E}^{1,0}_A$ is $C^2$ in the interior of $P$, in particular when $u \in \mathcal{E}^{1,\infty}_A$, it follows that 
	\begin{equation}\label{ding1-2}
		\F_1(u) = \int_P e^{-\rho_u}dx.
	\end{equation} 
\end{remark}
The Ding functional $\F$ on $\mathcal{E}^{1,0}_A$ is then defined to be

\begin{equation} \label{eqn3-8}
	\F (u) = \frac{1}{V_A}\int_P u e^{-A} dx - \frac{1}{2}\log{\F_1(u)}.
\end{equation}

\begin{lemma} \label{lem3-9}
	Suppose that $u \in \Pot$ satisfies \eqref{rhoequalsA} and that $w \in  C^0_0(\R^n)$ is a continuous and compactly supported (as a function on $\R^n$), such that $u_t = u + tw \in \mathcal{E}^{1,0}_A$ for sufficiently small $t$. Then the first variation of $\F_1$ at $u$ in the direction $w$ is given by
	
	\begin{equation*}
		\delta_{u}\F_1(w) = 2 \int_P w e^{-A(x)} dx,
	\end{equation*}
and consequently
	\begin{equation*}
		\left. \ddt \right|_{t=0} \F(u+tw) = 0. 
	\end{equation*}
\end{lemma}

\begin{proof}
	 Let $\phi_t = L(u_t)$.  Since $w$ is compactly supported,  $\nabla u: P \to \R^n$ is surjective, and the domain of the Legendre transform is convex, it follows that the domain of $\phi_t$ is the whole of $\R^n$, and moreover that $\phi_t = \phi_0 = L(u)$ outside of a fixed compact set independent of $t$.  Moreover,  by Lemma \ref{propernesslemma} we have that 
\begin{equation*}
	 \phi_t(\xi) \geq \varepsilon |\xi| - \sup_{B_{\varepsilon}(0)}| u_t| \geq \varepsilon |\xi| - C,
\end{equation*}
for $\varepsilon, C > 0$ independent of $t$.  By \eqref{eqn3-6}, we know that there is a set $E \subset P$ of measure zero and a compact subset $K_t \subset \R^n$ (which does not necessarily have zero measure) such that $\nabla u_t(P \backslash E) = \R^n \backslash K_t$ and $\sup_{\R^n\backslash K_t}|\dot{\phi}_t| = \sup_{P}|w| < \infty$.  This tells us two things. First, since $u_0 = u$ is smooth, we see that $K_0 = \nabla u(E) \subset \R^n$ has measure zero.  Moreover,  as we have seen the family $K_t$ is contained in a fixed ball $B \subset \R^n$ independent of $t$, so that in fact $\sup_{\R^n}|\dot{\phi}_t| \leq C(\sup_{P}|w|  + 1)$. Thus 
	\begin{equation*}
	\begin{split}
    |\dot{\phi}_t| e^{-\phi_t}  \leq C(\sup_{P}|w|  + 1) e^{-\phi_t}  \leq C(\sup_{P}|w|  + 1)e^{-\varepsilon|\xi| + C}  \in L^1(\R^n).  
	\end{split}
	\end{equation*}
Hence, by the mean value theorem and the dominated convergence theorem, it follows that 
	\begin{equation*} 
		\left.\ddt \F_1(u_t)\right|_{t=0} = -2 \int_{\R^n} \dot{\phi}_0  e^{-2\phi_0} d\xi = 2\int_P w e^{-A} dx,
	\end{equation*}
using \eqref{eqn3-6}, \eqref{rhoequalsA} and that $K_0$ has measure zero.  So

	\begin{align*}
		\left.\ddt \F_1(u_t)\right|_{t=0} &= \frac{1}{V_A}\int_P w e^{-A} dx - \frac{\delta_u \F_1(w)}{2\F_1(u)} \\
					& =  \frac{1}{V_A}\int_P w e^{-A} dx - \frac{1}{\int_{\R^n} e^{-2\phi_0}d\xi} \int_P w e^{-A} dx =  0,
	\end{align*}
	since $\int_{\R^n} e^{-2\phi_0}d\xi = \int_{P} e^{-\rho_u}dx =  \int_{P} e^{-A}dx$ by \eqref{rhoequalsA}. 
\end{proof}

\begin{proposition}[{c.f. \cite[Proposition 2.15]{BB}}] \label{prop3-10}
	The Ding functional $\F$ is convex on $\mathcal{E}^{1,0}_A$. It is invariant under the action of $\R^{n} \times \R$ given by addition of affine-linear functions, and it is strictly convex modulo this action. In particular, suppose that $u_0, u_1 \in \Pot_0$. Then if $\F(t u_1 + (1-t) u_0) = t\F(u_1) + (1-t)\F(u_0)$, there exists a linear function $\ell(x)$ on $P$ such that $u_1 = u_0 + \ell$. 
\end{proposition}

\begin{proof}
	If $u_0, u_1 \in \mathcal{E}^{1,0}_A$ satisfy $u_1 = u_0 + \ell(x) + a$ with $a \in \R$ and $\ell$ any linear function,  then by Lemma \ref{lem2-16} we see that $\int_{\R^n}e^{-\phi_1} d\xi = e^{2a} \int_{\R^n} e^{-\phi_0}d\xi $.  Therefore we see directly from the definition \eqref{eqn3-8} that the fact that $\F$ is invariant is equivalent to the statement that $\int_P \ell(x) e^{-A}dx = 0$ for any linear function $\ell$ on $P$, which $A$ satisfies by definition. We prove convexity directly, and show that 
	
	\begin{equation*}
		\F(u_t) \leq t\F(u_1) + (1-t) \F(u_0),
	\end{equation*}
where $u_t =tu_1 + (1-t)u_0$ for any $u_0, u_1\in \mathcal{E}^{1,0}_A$. Set $\phi_t = L(u_t)$. First notice that the functional $u \mapsto \int_P u e^{-A} dx$ is clearly affine on $\mathcal{E}^{1,0}_A$. Therefore it suffices to show that the function
	 
	\begin{equation*}
		t \mapsto -\log \int_{\R^n} e^{-2\phi_t} d\xi
	\end{equation*}
is convex in $t$. This follows from the fact that the Legendre transform is itself a convex mapping, i.e.
	
	\begin{equation} \label{eqn3-9}
		\phi_t(\xi) \leq t\phi_1(\xi) + (1-t) \phi_0(\xi),
	\end{equation}
which is the fourth item in Lemma \ref{lem2-16}.  It then follows immediately from the Pr\'{e}kopa-Leindler inequality \cite{Dub} that this is convex in $t$. This says precisely that any family $\phi_t$ of convex functions satisfying \eqref{eqn3-9} has the property that the function of one variable $\int_{\R^n}e^{-2\phi_t}d\xi$ is log-concave (i.e.  $ t \mapsto - \log \int_{\R^n} e^{-2\phi_t} d\xi$ is convex). The strict convexity follows from the equality case of the Pr\'{e}kopa-Leindler inequality, which was also studied in \cite{Dub}.  If the function $ \int_{\R^n}e^{-2\phi_t}d\xi$ is affine in $t$, then by \cite[Theorem 12]{Dub} there exists $m \in \R$ and $a \in \R^n$ such that 
	
	\begin{equation*} 
		\phi_1(\xi) = \phi_0(m\xi + a) - n\log(m) - \log\left( \frac{\int_{\R^n}e^{-2\phi_1} d\xi }{\int_{\R^n}e^{-2\phi_0} d\xi} \right).
	\end{equation*}
Firstly, we see that $m$ must be equal to 1 since $u_0, u_1 \in \Pot$. Indeed $L(\phi_0(m\xi)) = u_0(m^{-1}x)$. If $u_0 \in \Pot$, then $u_0(m^{-1}x) - u_P(x) \in C^\infty(\overline{P})$ if and only if $m = 1$. Then we have that $\phi_1(\xi) = \phi_0(\xi + a) - C$ for some $C$. Again passing to the Legendre transform, we have that 
	
	\begin{equation*}
		u_1(x) = L(\phi_1(\xi)) = L(\phi_0(\xi + a) - C) = u_0(x) + \ell_a(x) + C.
	\end{equation*}
Finally, the normalization condition \eqref{eqn3-5} implies that in fact $C = 0$. 
\end{proof}

To prove that solutions to \eqref{rhoequalsA} in $\Pot_0$ are unique, we would like to make use of this strict convexity.  To do this, we need to ensure that, if $u_0, u_1 \in \Pot_0$ are two solutions, the Ding functional $\F$ is minimized along the geodesic $u_t = tu_1 + (1-t)u_0$ at the endpoints $t = 0,1$. This would be clear from Lemma \ref{lem3-9} if the variation $v = u_1 - u_0$ were compactly supported, but there is no reason a priori why this should be the case.  To this end, we have 

\begin{lemma}\label{approximation-lemma}
 Suppose that $u \in \Pot_0$ and $v \in C^\infty(\overline{P})$ is such that $u_v := u +v \in \Pot_0$.  Then there exists a sequence of compactly supported functions $w_i \in C^0_0(\R^n)$ such that $U_i := u+ w_i \in \mathcal{E}^{1,0}_A$ and that 
 \begin{equation*}
 	\F(U_i) \to \F(u_v)
 \end{equation*}
 as $i \to \infty$. 
\end{lemma} 

\begin{proof}
Let $\t^*_C \subset \R^n$ be the linear subspace spanned by recession cone $C$ of $P$.  We can see from the definition of $C$ (Definition \ref{defi2-3}) that there exists some point $q \in \R^n$,  not necessarily unique,  such that the translate $C - q$ coincides with the intersection $P_C$ of $P$ with $\t^*_C$. for each $k \geq 0$ set $B_{k}^C$ to be the cylinder 
 	\begin{equation*}
 		B_k^C = \{ x \in \R^n \: | \: || x - q ||_{\t_C} < k \}, 
 	\end{equation*}
 where $||  \cdot ||_{\t_C}$ denotes the norm of the induced inner product on $\t_C$.  As a shorthand we will denote $r(x) = || x - q ||_{\t_C}$. Note that, if we set $\Omega_k = \overline{P} \cap B_k^C$,  then any point in $\Omega_k$ can be joined to $\Omega_{\tilde{k}}$ by a line emanating from $q$ for any $k, \tilde{k}$ sufficiently large.  Now, $u_v$ is proper, so we can choose a $k_1 \geq 0$ sufficiently large such that the set $\Omega_{k_1}$ contains the unique critical point of $u_v$.  Let $\alpha_1 = \sup_{\p\Omega_{k_1}} \frac{\p u_v}{\p r} + 1$,  noting that this is finite by the choice of $\Omega_{k_1}$. Indeed, by construction we have that $\frac{\p}{\p r}$ is tangent to any face of $P$, and hence the corresponding quantity $ \sup_{\p\Omega_{k_1}} \frac{\p u_P}{\p r}$ for $u_P$ is finite.  Set $\tilde{u}_{v,1}$ to be continuous convex function on $\overline{P}$ defined by setting $\tilde{u}_{v,1} = u_v$ on $\Omega_{k_1}$ and extending continuously linearly with slope $\alpha_1$, i.e.  
 \begin{equation*}
 	\tilde{u}_{v,1}(y) = \left\{  \begin{array}{cl}  u_v(y) & y \in \Omega_{k_1} \\ u_v\left( \pi_{k_1}(y) \right) + \alpha_1 r(y - \pi_{k_1}(y)) &  y \in \overline{P} \backslash \Omega_{k_1} \end{array} \right.  ,
 \end{equation*}
  where $\pi_{k_1}(y) = y - (1 -  \frac{k_1}{r(y)}) (y - q) $ is the linear projection onto $\p\Omega_{k_1}$ relative to the base point $q$. Since $u$ grows faster than linearly in $|x|$ by Lemma \ref{surjective-gradient},  we can choose $k_2$ sufficiently large such that $u \geq \tilde{u}_{v,1} + 1$ on $\overline{P} \backslash \Omega_{k_2}$, $\inf_{\p \Omega_{k_2}} \frac{\p u}{\p r} \geq \alpha_1 + 1$.  We then choose $\beta_1 = \inf_{\p \Omega_{k_2}}  - 1$ and set $\tilde{u}_{1}$ to be 
   \begin{equation*}
 	\tilde{u}_{1}(y) = \left\{  \begin{array}{cl}  u\left(\pi_{k_2}(y) \right) - \beta_1 r(y - \pi_{k_2}(y)) & y \in \Omega_{k_2} \\ u(y)&  y \in \overline{P} \backslash \Omega_{k_2} \end{array} \right.  .
 \end{equation*}
  By construction,  $\tilde{u}_{1}(y) \geq \tilde{u}_{v,1}(y)$ on $\p \Omega_{k_2}$.  As a consequence of the tangent plane property of convexity, the properness of $u$,  together with the monotonicity of $\frac{\p u}{\p r}$, we see that the norm $|y|$ (equivalently $||y||_C$) of any point satisfying $ u\left(\pi_{k_2}(y) \right) - \beta_1 r(y - \pi_{k_2}(y)) = u_v\left( \pi_{k_1}(y) \right) + \alpha_1 r(y - \pi_{k_1}(y))$ can be made to strictly increase by sufficiently increasing the value of $k_2$. 
  Hence after perhaps making an even larger choice for $k_2$ we can ensure that the set of points $y$ such that $\tilde{u}_{1}(y) =  \tilde{u}_{v,1}(y)$ lies inside (the closure of) of $\Omega_{k_2} \backslash \Omega_{k_1}$. 
  Thus,  if we set $U_1 = \textnormal{max}\{\tilde{u}_{1} ,  \tilde{u}_{v,1}\}$, then $U_1$ is convex and 
 \begin{equation*}
 	U_1(x) = \left\{\begin{array}{cl} u_v(x) & x \in \Omega_{k_1} \\ u(x) & x \in \overline{P} \backslash \Omega_{k_2}  \end{array} \right.  .
 \end{equation*}
In particular,  if we set $w_1 = U_1 - u$, we see that $w_1 \in C^0_0(\R^n)$ has support in $\Omega_{k_2}$. Continuing in this way, we produce a sequence of functions $w_i \in C^0_0(\R^n)$ together with a sequence of compact convex sets $\Omega_{k_i}$ such that $U_i = u + w_i$ is convex, $w_i = v$ on $\Omega_i$ and $w_i = 0$ on $\overline{P}\backslash \Omega_{i+1}$.  Moreover, it follows from the construction that in fact $U_i \leq \textnormal{max}\{u, u_v\}$ everywhere.  
 	
 	Now since $U_i = u$ outside of a compact set, we see that $\int_P |U_i| e^{-A} dx < \infty$, and consequently $U_i \in \mathcal{E}^{1,0}_A$. In order to deduce that $\lim_{i \to \infty} \F(U_i) = \F(u_v)$, we first argue that $\lim_{i \to \infty} \int_P U_i e^{-A}dx = 0$.  For any $\varepsilon > 0$, let $i_0$ be sufficiently large such that 
 	\begin{equation*}
 	\left|\int_{\Omega_{i}} u_v  e^{-A} dx \right| + \left|\int_{P \backslash \Omega_{i}} u_v  e^{-A} dx \right| +\left|\int_{P \backslash \Omega_{i}} u e^{-A} dx \right|< \varepsilon, 
 	\end{equation*}
 	for all $i \geq i_0$. Clearly we can increase $i_0$ if necessary to ensure that $U_i, u, u_v \geq 0$ on $P \backslash \Omega_i$ for all $i \geq i_0$.  Hence for $i \geq i_0$ we have
 	\begin{equation*}
 	\begin{split}
 		\left|\int_P U_{i} e^{-A} dx  \right| &\leq \left|\int_{\Omega_i} u_v  e^{-A} dx \right| + \left| \int_{\Omega_{i+1}\backslash \Omega_i} U_i e^{-A} dx  \right| + \left| \int_{P\backslash \Omega_{i+1}} u e^{-A}dx \right|  \\
 		& \leq \varepsilon + \left|  \int_{\Omega_{i+1}\backslash \Omega_i}U_i e^{-A} dx \right| = \varepsilon +   \int_{\Omega_{i+1}\backslash \Omega_i}U_i e^{-A} dx  \\
 			& \leq \varepsilon +  \int_{P \backslash \Omega_i} \max\{u, u_v\} e^{-A} dx =  \varepsilon +  \int_{(P \backslash \Omega_i) \cap \{u \leq u_v\} } \!\!\!\!\!\!\!\!\!\!\!\!\!\!\!\!\!u_v e^{-A} dx +  \int_{(P \backslash \Omega_i) \cap \{u \geq u_v\} } \!\!\!\!\!\!\!\!\!\!\!\!\!\!\!\!\!u e^{-A} dx \\
 			 & \leq  \varepsilon +  \int_{P \backslash \Omega_i} u e^{-A} dx  +  \int_{P \backslash \Omega_i} u_v e^{-A} dx \leq 2 \varepsilon.
 	\end{split}
 	\end{equation*}
	Next, we claim that $ \lim_{i \to \infty}\int_{\R^n} e^{-2\phi_i} d\xi = \int_{\R^n} e^{-2\phi_v} d\xi$.  Once again fix some $\varepsilon > 0$, and set $\phi_i = L(U_i)$, $\phi = L(u), \phi_v = L(u_v)$.  By Lemma \ref{propernesslemma}, we have that 
	\begin{equation*}
		\phi_i(\xi) \geq \delta |\xi| - \sup_{B_{\delta}(0)} |U_i| = \delta |\xi| - \sup_{B_{\delta}(0)} u_v \geq \delta |\xi| - C, 
	\end{equation*}
 for some fixed $\delta > 0$ sufficiently small, and uniformly for all $i$ sufficiently large.  Since $\phi_v$ is proper, perhaps after modifying  $C$ we can ensure that $\phi_v \geq \delta |\xi| - C$ for the same choice of $\delta$ and $C$. Next choose $R > 0$ sufficiently large such that $e^{2C}\int_{\R^n \backslash B_{R}(0)}e^{-2\delta |\xi|} d\xi < \varepsilon$, and then $i_0$ sufficiently large that $B_R(0) \subset \nabla u_v(\Omega_{i_0})$, which we can achieve by Lemma \ref{surjective-gradient}. Then since $u_i = u_v$ on $\Omega_{i}$,  it follows that $\nabla u_i = \nabla u_v$ on the interior of $\Omega_i$ and hence $\phi_i = \phi_v$ on $\nabla u_v (\Omega_i)$.  Thus
 \begin{equation*}
 \begin{split}
 	\left| \int_{\R^n} e^{-2\phi_i} d\xi - \int_{\R^n} e^{-2\phi_v} d\xi \right| &\leq \left| \int_{\nabla u_v(\Omega_i)} \left( e^{-2\phi_i}  - e^{-2\phi_v}\right) d\xi \right| +  \int_{\R^n \backslash \nabla u_v(\Omega_i)} \left(e^{-2\phi_i}  + e^{-2\phi_v} \right) d\xi  \\
 		& \leq \int_{\R^n \backslash B_R(0)} \left(e^{-2\phi_i}  + e^{-2\phi_v} \right) d\xi \leq 2e^{2C}\int_{\R^n \backslash B_{R}(0)}e^{-2\delta |\xi|} d\xi  \leq 2\varepsilon ,
 \end{split}
 \end{equation*}
 for all $i \geq i_0$. Thus $\int_{\R^n} e^{-2\phi_i} d\xi \to \int_{\R^n} e^{-2\phi_v} d\xi$ as desired, and finally we conclude that $\F(u_i) \to \F(u_v)$. 
\end{proof}

 \begin{theorem} \label{thm3-11}
 	Let $P$ be a polyhedron containing zero in its interior, and suppose that $A \in C^\infty(\overline{P})$ is admissible. Then up to the action of the linear functions, there is at most one solution $u$ to \eqref{rhoequalsA} in $\Pot_0$. 
 \end{theorem} 
 
 \begin{proof}
 	Suppose that we have two solutions $u_0, u_1 \in \Pot_0$, and let $u_t = tu_1 + (1-t)u_0$, $v = u_1 - u_0$.  Fix any $t \in (0, 1)$. By Lemma \ref{approximation-lemma}, there exists a sequence of compactly supported functions $w_i$ such that $U_i = u_0 + w_i \in \mathcal{E}^{1,0}_A$ and $\F(U_i) \to \F(u_t)$. By Lemma \ref{lem3-9} and Proposition \ref{prop3-10},  moreover, we know that $\F(U_i) \geq \F(u_0)$, and therefore by passing to the limit we see that $\F(u_t) \geq \F(u_0)$.  Of course this is completely symmetric in $u_0$ and $u_1$ and independent of the choice of $t$,  and hence it follows that $\F(u_t)$ is minimized at $t = 0$ and $t = 1$.  Now let $\mathcal{H}$ denote the space of equivalence classes $[u]$ in $\Pot_0$ under the action of $\R^n$ by the addition of linear functions. By Proposition \ref{prop3-10}, $\F$ descends to a strictly convex functional on $\mathcal{H}$, and we have just seen that the convex function of one variable
	\begin{equation*}
		t \mapsto \F\left([u_t]\right)
	\end{equation*}
is minimized at both $t = 0$ and $t = 1$, and hence is constant. Since $\F$ is strictly convex, it follows that $[u_0] = [u_1]$.
 \end{proof}

\section{Proofs of the main theorems}

\subsection{Preliminaries}

Let $(M,J)$ be a complex manifold with a fixed effective and holomorphic action of the real torus $T^{n}$ with finite fixed point set. Suppose that $\omega$ is the K\"ahler form of a complete shrinking gradient K\"ahler-Ricci soliton $(g,X)$ on $M$ with $JX \in \t$. By \cite[Theorem 1.1]{Wylie} it follows that any manifold which admits a complete shrinking Ricci soliton must satisfy $H^1(M) = 0$. It is an immediate consequence that the $T^n$ action is Hamiltonian with respect to the K\"ahler form $\omega$ of $g$. Indeed, let $X_1, \dots, X_n$ be any basis for $\t$, and $\theta_j \in C^\infty(M)$ satisfy $-i_{X_j}\omega = d \theta_j$. Then one defines a moment map explicitly by the formula $\mu(x) = \left(\theta_1, \dots, \theta_n \right)$. There is of course an ambiguity in the choice of each $\theta_j$ of the addition of a constant. Put together, this corresponds to a translation of the image $\mu(M) \subset \t^*$. We begin by showing that if we assume that the Ricci curvature of $g$ is bounded, we can fit this situation into the general framework of the previous sections.

\begin{lemma}\label{lem4-1}
	Let $(M,J, \omega)$ be as above, and suppose that $g$ has bounded Ricci curvature and that $JX \in \t$. Then there exists a complexification of the $T^{n}$-action, i.e. an action of $\Cstarn$ whose underlying real torus corresponds to the original $T^{n}$-action. Furthermore, there exists an automorphism $\alpha$ of $(M,J)$ such that $\alpha^{*}g$ is $T^{n}$-invariant.
\end{lemma}

To prove this, we make use of the general structure theory for holomorphic vector fields on manifolds admitting K\"ahler-Ricci solitons from \cite{ConDerSun}. Let $\mathfrak{aut}^X$ be the space of holomorphic vector fields commuting with the soliton vector field $X$ and $\mathfrak{g}^X$ be those real holomorphic killing fields commuting with $X$. 

\begin{theorem}[{\cite[Theorem 5.1]{ConDerSun}}] \label{thm4-2} Let $(M,J,g,X)$ be a complete shrinking gradient K\"ahler-Ricci soliton with bounded Ricci curvature such that $JX \in \t$. Then 

	\begin{equation}
		\mathfrak{aut}^X = \mathfrak{g}^X \oplus J\mathfrak{g}^X
	\end{equation}
Furthermore, $\mathfrak{aut}^X$ and $\mathfrak{g}^X$ are the Lie algebras of finite-dimensional Lie groups $Aut^X$ and $G^X$ corresponding to holomorphic automorphisms and holomorphic isometries commuting with the flow of $X$. 

\end{theorem}

\begin{proof}[of Lemma 4.1]
Let $(X_1, \dots, X_n)$ be a basis for $\t$. Since $JX \in \t$, it is clear that $[X, X_i] = [X, JX_i] = 0$ for any $i$. In particular, $\t \subset \mathfrak{aut}^{X}$. Since the scalar curvature of $g$ is bounded by assumption, we have by \cite[Lemma 2.26]{ConDerSun} that the zero set of $X$ is compact. Therefore by \cite[Lemma 2.34]{ConDerSun}, it follows that for each $i$, $X_i$ and $JX_i$ are complete. In particular, the flow of $(X_i, JX_i)$ determines a unique effective and holomorphic action of $\Cstar$. Thus we can complexify the $T^{n}$ action, and moreover the corresponding $\Cstarn$-action satisfies $\t_\C = \t \oplus J\t \subset \mathfrak{aut}^{X}$. Since then $X$ and $JX$ lie in $\mathfrak{aut}^{X}$, we have that the $\Cstarn$-action on $M$ embeds $\Cstarn \subset Aut^X$, and so the real torus $T^n \subset \Cstarn$ lies in some maximal compact subgroup $G$ of $Aut^X$. Since any two maximal compact subgroups of a reductive group are conjugate by Iwasawa's theorem \cite{Iwa}, it follows such that there exists an automorphism $\alpha$ such that the group $G$, and therefore $T^n$, preserves the metric $\alpha^*g$. 
\end{proof}

Thus, for the remainder of this section, we assume that $(M,J)$ admits an effective holomorphic $\Cstarn$-action with finite fixed point set, and $\omega$ is the K\"ahler form of a complete $T^{n}$-invariant shrinking gradient K\"ahler-Ricci soliton $(g,X)$. In particular, if there is an element $b \in \t$ such that $\langle \mu , b \rangle$ is proper and bounded from below, then $M$ is AK-toric by Lemma \ref{lem2-20}. We have by Proposition \ref{prop2-10} that there exists a potential $\phi$ for $\omega$ on the dense orbit which can be viewed as a smooth strictly convex function on $\R^n$. We note also that $\omega$ is the curvature form of the $T^n$-invariant hermitian metric $h_{X} = e^{-f}(\omega^n)^{-1}$ on $-K_M$. From \eqref{eqn2-4} we know that the soliton potential $f$ is given by 

	\begin{equation*}
		f = \langle \nabla\phi, b_X \rangle = \langle \mu, b_X \rangle.
	\end{equation*}
We have the following from \cite{CaoZhou}.

\begin{proposition}[{\cite[Theorem 1.1]{CaoZhou}}]\label{prop4-3}
	Let $(M, g, f)$ be any non-compact complete shrinking gradient Ricci soliton. The soliton potential $f$ grows quadratically with respect to the distance function $d_g$ defined by $g$, so there is a constant $c_f$ such that 
		
			\begin{equation*}
				\frac{1}{4}(d_p - c_f)^2 \leq f \leq \frac{1}{4}(d_p + c_f)^2.
			\end{equation*}
		
\end{proposition}

Therefore $b_X \in \t$ is an element for which the map $\langle \mu , b_X \rangle: M \to \R$ is proper and bounded from below. Thus $\mu$ has image equal to a Delzant polyhedron $P$ by Lemma \ref{lem2-20}, and therefore $M$ is AK-toric. Let $\{D_i\}_{i=1, \dots, m}$ be the collection of prime, $\Cstarn$-invariant divisors in $M$. Since the anticanonical divisor $-K_M$ of a toric variety is always given by the simple formula \cite[Theorem 8.2.3]{CLS} 

\begin{equation*}
	-K_M \sim \sum_{i=1}^m D_i, 
\end{equation*}
we can apply Corollary \ref{cor2-25} to obtain:

\begin{lemma} \label{lem4-4}
	Let $(M,J)$ be a complex manifold with an effective holomorphic $\Cstarn$-action with finite fixed point set. Suppose that $\omega$ is the K\"ahler form of a complete $T^{n}$-invariant shrinking gradient K\"ahler-Ricci soliton $(g,X)$ on $M$. Then the moment map $\mu$ has image equal to a Delzant polyhedron $P$. In particular, $(M,J,\omega)$ is AK-toric and quasiprojective. Let $\{D_i\}$ be the prime, $\Cstarn$-invariant divisors in $M$, and let $\nu_i \in \Z^n \subset \t$ be minimal generators of the corresponding rays given by the Orbit-Cone correspondence. Then the image $P$ of $\mu$ is equal up to translation to the polyhedron 
	
	\begin{equation}\label{eqn4-2}
		P_{-K_M} = \left\{ x \in \t^* \: | \: \langle \nu_i, x \rangle \geq -1 \right\}
	\end{equation}
determined by the anticanonical bundle. 
\end{lemma}

In particular, the line bundle $L_P$ of Proposition \ref{prop2-14} is equal to $-K_M$. Clearly, zero lies in the interior the polyhedron $P_{-K_M}$ above whenever it is full-dimensional. For simplicity of notation, we will denote $P = P_{-K_M}$.

We emphasize that as yet the image of the moment map is fixed only up to translation in $\t^*$. Recall (Lemma \ref{lem2-18}) that the addition of a linear function to the K\"ahler potential $\phi = \phi(\xi)$ on the dense orbit corresponds to a translation of the image of the moment map. We claim that the normalization determined in Proposition \ref{prop2-12} fixes the moment image uniquely. Thus, it is the real Monge-Amp\`ere equation \eqref{eqn2-5} that fixes which translate of $P \subset \t^{*}$ appears. The argument is local, and is based on the observation of Donaldson \cite{Don1} that the choice of normalization for $\phi$ determines the behavior of K\"ahler-Ricci soliton equation \eqref{gkrs} in symplectic coordinates as $x \to \partial P$. 
 
 \begin{lemma} \label{lem4-5}
 	Let $(M,J,\omega)$ be AK-toric, and suppose that $\omega$ is the K\"ahler form of a complete shrinking gradient K\"ahler-Ricci soliton on $M$. Then, by Proposition \ref{prop2-12}, there exists a unique smooth convex function $\phi$ on $\R^n$ such that $\phi$ determines a K\"ahler potential for $\omega$ on the dense orbit via the identification $\Cstarn \cong \R^n \times T^n$ and satisfies the real Monge-Amp\`ere equation 
	
	\begin{equation*}
		\det\phi_{ij} = e^{-2\phi + \langle b_X, \nabla \phi \rangle}.
	\end{equation*}
Then the image of the moment map $\mu = \nabla \phi$ is precisely the translate of $P$ given in \eqref{eqn4-2}. In particular, zero lies in the interior of $P$.
	
 \end{lemma}
 
 \begin{proof}
  	We know from Lemma \ref{lem4-4} that the image $\nabla\phi(\R^n)$ is a Delzant polyhedron $P'$. Suppose that $P'$ is defined by the linear inequalities $\ell_i(x) \geq -a_i$, where $\ell_i(x) = \langle \nu_i , x \rangle$. As we saw in Proposition \ref{prop2-26}, any such $\omega$ determines and is determined by a symplectic potential $u \in C^\infty(P)$, which is unique up to the addition of an affine function. Passing to the Legendre transform, recall that $u$ satisfies the real Monge-Amp\`ere equation $\rho_u = \langle b_X, x \rangle$, where 
	
	\begin{equation*}
		\rho_u(x) = 2 \left( u_i x^i - u\right) - \log\det(u_{ij}).
	\end{equation*}
In particular, $\rho_u$ extends smoothly past $\partial P$. By Proposition \ref{prop2-26}, there exists a function $v$ on $P$, extending smoothly across $\partial P$, such that $u = u_P + v$, where $u_P$ is defined as in \eqref{eqn2-10} by 
	
	\begin{equation*}
		u_P(x) = \frac{1}{2} \sum (\ell_i(x) + a_i) \log(\ell_i(x) + a_i).
	\end{equation*}
Fix any facet $F$ of $P'$. We may assume that $F$ is given by $\ell_1(x) = -a_1$. Up to a change of basis in $\t^*$, we may also assume by the Delzant condition that $\ell_1(x) = x_1$. Choose a point $p$ in the interior of $F$. Near $p$, $u_P$ can therefore be written 
	
	\begin{equation*}
		u_P(x) = \frac{1}{2} (x_1 + a_1)\log(x_1 + a_1) + v_1,
	\end{equation*}
where $v_1$ extends smoothly across $F$. It then follows that in a small half ball $B$ in the interior of $P'$ containing $p$, $\rho_u$ can be expressed as 
	
	\begin{equation*}
		\rho_u(x) = x_1 \log(x_1 + a_1) -(x_1 + a_1)\log(x_1 + a_1) + \log(x_1 + a_1) + v_2,
	\end{equation*}
where $v_2$ again extends smoothly across $F$ in $B$. It follows that $a_1 = 1$. 
 \end{proof}
 
 In the compact case, the condition that $M \cong \M_P$ for $P$ given by \eqref{eqn4-2} is equivalent to the condition that $M$ is Fano. We therefore make the following definition.
 
 \begin{definition}[{c.f. \cite[Definition 7.1]{ConDerSun}}]\label{def4-6} Let $M$ be a complex toric manifold. We say that the pair $(M, -K_M)$ is \emph{anticanonically polarized} if $M \cong \M_{P_{-K_M}}$. 
 \end{definition}
 
In particular, an anticanonically polarized toric manifold is quasiprojective.

\begin{theorem}\label{thm4-7}
	There exists a unique holomorphic vector field $X$ with $JX \in \t$ on an anticanonically polarized AK-toric manifold $(M, -K_M)$ which could be the vector field of a complete $T^n$-invariant shrinking gradient K\"ahler-Ricci soliton. 
\end{theorem}

 \begin{proof}
 Let $\omega_1$ and $\omega_2$ be two $T^{n}$-invariant K\"ahler metrics on $M$ satisfying \eqref{gkrs} on $M$ with vector fields $X_1$ and $X_2$. By Lemma \ref{lem4-5}, we know that each moment map $\mu_s$, $s=1,2$, has image equal to $P = P_{-K_M}$. Moreover, by Lemma \ref{lem4-5}, we know that $\omega_s$ is uniquely determined by a symplectic potential $u_s$ on the fixed polyhedron $P = P_{-K_M}$ which satisfies the real Monge-Amp\`ere equation $\rho_{u_s} = \langle b_{s}, x \rangle$. By Lemma \ref{lem3-6}, the function $\langle b_{s}, x \rangle$ satisfies
	
	\begin{equation*}
		\int_P \ell(x) e^{-\langle b_{s}, x \rangle} dx = 0
	\end{equation*}
for each linear function $\ell(x)$ on $P$. In particular, $\langle b_{s}, x \rangle$ is equal to the fixed linear function $\ell_P$ determined in Proposition \ref{prop3-1}. Clearly, there is a unique $b_P \in \t$ such that $\ell_P(x) = \langle b_P, x \rangle$. Let $X_P$ be the holomorphic vector field on $M$ which is given by
	\begin{equation*}
		X_P^{1,0} = \sum_{j = 1}^n b_P^{j} z_j\frac{\p}{\p z_j}
	\end{equation*}
on the dense orbit. We have in particular that $\mathcal{L}_{X_P}\omega_s = \mathcal{L}_{X_s} \omega_s$. Since $\omega_s$ is $T^{n}$-invariant and $JX_P, JX_1, JX_2 \in \t$, this immediately implies that $X_1 = X_2 = X_P$. 
\end{proof}

 \subsection{Proofs of Theorem A and Theorem B}

We begin with the proof of Theorem \ref{thm1-1}. Suppose that $\omega_1$ and $\omega_2$ are two complete $T^n$-invariant K\"ahler metrics on $M$ satisfying \eqref{gkrs}. By Theorem \ref{thm4-7}, the soliton vector fields are given by $X_1 = X_2 = X_P$. Recall from the proof of Theorem \ref{thm4-7} we know that each $\omega_s$ is determined uniquely by a symplectic potential $u_s$ on the fixed polyhedron $P$. Each $u_s$ itself is unique up to the addition of an affine function, and satisfies the real Monge-Amp\`ere equation

\begin{equation} \label{eqn4-3}
	\rho_{u_s} = \langle b_P, x \rangle,
\end{equation}
where $b_P \in \t$ is the element determining $X_P$ as in the proof of Theorem \ref{thm4-7}. If we set 

\begin{equation*}
	A(x) = \langle b_P, x \rangle,
\end{equation*}
then equation \eqref{eqn4-3} takes the form $\rho = A$ with respect to the fixed function $A$ on $P$. Thus, we are in the setting of Section 3.3. We would then like to apply the uniqueness theorem Theorem \ref{thm3-11} to conclude that $u_s$ are related via the addition of an affine function. We need to show therefore that $A$ is admissible and that $\int_P u e^{-A}dx < \infty$, so that $u_s$ lies in the space of symplectic potentials $\Pot$ defined by $A$. To see that $A$ is admissible, first note that by Lemma \ref{lem3-6}, we have

	\begin{equation*}
		\int_P \ell  e^{-A}dx = 0,
	\end{equation*}
which is condition $(2)$ from Section 3.3. Since, by Proposition \ref{prop4-3}, 

	\begin{equation*}
		\int_M e^{-f} \omega^n < \infty,
	\end{equation*}
we have that

	\begin{equation*}
		\int_M e^{-f} \omega^n = \int_{\R^n \times T^n} e^{-\langle b_P, \nabla \phi \rangle} \det(\phi_{ij}) d\xi d\theta = (2\pi)^n \int_P e^{- A}dx.
	\end{equation*}
This implies that 
	
	\begin{equation}\label{eqn4-4}
		\int_P e^{-A}dx < \infty,
	\end{equation}
which is condition $(1)$. Furthermore, from \eqref{eqn4-4} it follows from Corollary \ref{cor3-2} that $b_P \in C^*$, and in particular $A(x) = O(|x|)$. Since $u_P = O(|x|\log|x|)$ we then have 

	\begin{equation*}
		\int_P u_P e^{-A} dx < \infty,
	\end{equation*}
which is condition $(3)$. Thus $A$ is admissible, and it remains only to show that each $\int_P u_s e^{-A} < \infty$. This follows from an elementary calculation.

\begin{lemma}[{c.f. \cite[Lemma 1]{Don1}}]  \label{lem4-8}
	Let $P$ be a polyhedron containing zero in the interior and $u \in C^\infty(P)$ be any strictly convex function such that the gradient $\nabla u$ maps $P$ diffeomorphically onto $\R^n$. Then 
	
	\begin{equation*}
		\int_P u e^{-\rho_u}dx < \infty.
	\end{equation*}
	
\end{lemma} 

\begin{proof}
	Let $\phi(\xi) = L(u)$. Recall that by Lemma \ref{propernesslemma}, $\phi$ grows at least linearly in $|\xi|$, and in particular is necessarily bounded from below. Then 
	
	\begin{equation*}
		\int_P u e^{-\rho_u}dx = \int_{\R^n} \left(  \langle \nabla \phi, \xi \rangle - \phi \right) e^{-2\phi} d\xi \leq \int_{\R^n } \left( \langle \nabla \phi, \xi \rangle + C \right) e^{-2\phi} d\xi.
	\end{equation*}
The second term $\int C e^{-2\phi}d\xi$ is bounded again by Lemma \ref{propernesslemma}, so that 
	
	\begin{equation*}	
		\int_P u e^{-\rho_u}dx \leq \int_{\R^n} \langle \nabla \phi, \xi \rangle e^{-2\phi} d\xi  + C.
	\end{equation*}
In polar coordinates we have 
	
	\begin{equation*}
		\int_{\R^n} \langle \nabla \phi, \xi \rangle e^{-2\phi} d\xi  =  \int_{S^{n-1}} \int_0^\infty r^n \frac{\p \phi}{\p r} e^{-2\phi} dr d\Theta.
	\end{equation*}
Integrating by parts, we obtain
	
	\begin{equation*}
		\int_{S^{n-1}} \int_0^\infty r^n \frac{\p \phi}{\p r} e^{-2\phi} dr d\Theta = \frac{n}{2}\int_{S^{n-1}} \int_0^\infty r^{n-1}  e^{-2\phi} dr d\Theta = n \int_{\R^n} e^{-2\phi} d\xi .
	\end{equation*}
Note that the boundary term converges since $\phi = O(r)$ as $r \to \infty$. Thus 
	
	\begin{equation*}
		\int_P u e^{-\rho_u}dx \leq \int_{\R^n} \langle \nabla \phi, \xi \rangle e^{-2\phi} d\xi  + C  =n \int_{\R^n} e^{-2\phi} d\xi  + C  < \infty.
	\end{equation*}
\end{proof}

Since each $u_s$ satisfies $\rho_{u_s} = A$, Lemma \ref{lem4-8} states that 

\begin{equation}\label{eqn4-5}
	\int_P u e^{-A} dx < \infty.
\end{equation}
Each $u_s$ is strictly convex on $P$, and by Proposition \ref{prop2-26} there exists for each $u_s$ a smooth function $v_s \in C^{\infty}(\overline{P})$ such that $u_s = u_P + v_s$. Strict convexity of $u_s$ along with \eqref{eqn4-5} then imply that $u_s \in \Pot$, and so by Theorem \ref{thm3-11} it follows that there is an affine function $a(x) = \langle b_a, x \rangle + c$ such that $u_2 = u_1 + a$. Let $\phi_s = L(u_s)$ be the Legendre transform, so that $\omega_s =2 i \p \bp \phi_s(\xi)$ on the dense orbit. As we have seen in Lemma \ref{lem2-18}, it follows that $\phi_2(\xi) = \phi_1(\xi - b_a) -c$, so that $2i \p \bp \phi_2(\xi) = 2 i \p \bp \phi_1(\xi - b_a)$. Let $\alpha: M \to M$ denote the automorphism determined by the action of $e^{-b_a} \in \Cstarn$. Then it is clear that $\phi_1(\xi - b_a) = \phi_1 \circ \alpha(\xi)$, and therefore that $\omega_2 = \alpha^{*}\omega_1$. This concludes the proof of Theorem \ref{thm1-1}.

Theorem \ref{thm1-2} follows immediately from Lemma \ref{lem4-1}, Lemma \ref{lem4-4}, and Theorem \ref{thm1-1}.

\subsection{Proof of Corollary \ref{cor1-4}}
 Recalling the setting, let $N$ be an $(n-1)$-dimensional compact toric Fano manifold, and $L \to N$ satisfy $L^p = K_N$ for $0 < p < n$. By Theorem \ref{thm1-2}, it suffices to show that the metrics have bounded Ricci curvature and that the corrsponding soliton vector fields satisfy $JX \in \t$.  We first observe that the total space of $L$ admits an effective and holomorphic $\Cstarn$-action by augmenting the $(n-1)$-dimensional action on $N$ with the natural $\Cstar$-action acting on the fibers of $L$.  It was shown in \cite{FOW} that the cone formed by contracting the zero section on $L$ admits a Ricci-flat K\"ahler cone metric $\omega_{RF} = \frac{i}{2}  \p \bp \tilde{r}^2$ with Reeb vector field $J \tilde{r} \frac{\p}{\p \tilde{r}} = K \in \t$.  Futaki's construction begins by deforming $\omega_{RF}$ to what's called a \emph{Sasaki $\eta$-Einstein metric} by a choice of reparameterization of the radial function $\tilde{r} \mapsto r = \tilde{r}^a$ for some $a >0$ (here $\eta = d^c \log r$ refers to the contact 1-form associated to the Sasakian structure).  Set $\omega = \frac{i}{2} \p \bp r^2$ to be this choice and set $t = \log r$ and $\omega_T = i \p \bp t$.  Then the metric $\omega_{KRS}$ is chosen via the \emph{momentum construction} (or \emph{Calabi Ansatz}), and thus splits orthogonally as
 \begin{equation*}
 \begin{split}
   \omega_{KRS} = \omega_T + i \p \bp H(t) = (1 + \tau) \omega_T + \varphi(\tau) dt \wedge d^c t,
 \end{split}
 \end{equation*}
where $H$ is a smooth convex function of one variable,  $\tau = H'(t)$, $\varphi(\tau) = H''(t)$.  Here $\tau \in (0, \infty)$ and $\tau \to 0$ corresponds to approaching the zero section of $L$ whereas $\tau \to \infty$ goes off to infinity along the complete end. We refer to \cite{FOW,FutWang,Fut} (see also \cite{FIK,HS}) for more details on this construction. In particular, the soliton vector field satisfies $JX = r \frac{\p}{\p r} \in \t$.

To see that the Ricci curvature of $\omega_{KRS}$ is bounded, we use the explicit form  \cite[Claim 4.4]{Fut} of $\varphi$
\begin{equation*}
	\varphi(\tau) = \frac{(\kappa - 2)}{\mu}(1 + \tau) + \frac{\kappa - 2 - \frac{\kappa }{n}}{\mu^{n+1}}\sum_{j = 0}^{n-1}\frac{n!}{j!}\mu^j (1 + \tau)^{j-(n-1)},
\end{equation*}
where $\kappa > 2, \mu > 0$ are constants determined by the soliton equation.  So $\varphi$ is a rational function and one sees immediately that $\varphi = O(1+\tau),  \varphi' = O(1), \varphi'' = O((1+\tau)^{-3})$ as $\tau \to \infty$.  Moreover, the Ricci form is also explicit (\cite[Equation 3.8]{Fut}) 
\begin{equation*}
	\Ric_{\omega_{KRS}} = \left( \kappa - \left(  \frac{(n-1)\varphi}{1 + \tau} + \varphi' \right)\right)\omega_T - \left(  \frac{(n-1)\varphi}{1 + \tau} + \varphi' \right)' dt \wedge d^c t.
\end{equation*}
Thus we read off that $\Ric_{\omega_{KRS}} = O(1)\omega_T + O((1+\tau)^{-2}) dt \wedge d^c t$, whereas the metric $\omega_{KRS} = O(1+\tau) \omega_T + O(1+\tau)dt \wedge d^c t$, from which we see that $||\Ric_{\omega_{KRS}}||_{\omega_{KRS}}$ actually decays as $\tau \to \infty$. 
\begin{flushright}
$\square$
\end{flushright}

\subsection{Example: $ \C\P^{1} \times \C$ }

Choose homogeneous coordinates $[w_1 : w_2]$ on $\C\P^{1}$, and let $w = \frac{w_1}{w_2}$. We let $\Cstar$ act on $\C\P^{1}$ by $\lambda \cdot [w_1 : w_2] = [\lambda w_1 : w_2]$, which gives $\C\P^{1}$ the structure of a toric variety. Let $\omega_{FS}$ be the Fubini-Study metric associated to $[w_1 : w_2]$. Let $z$ be a holomorphic coordinate on $\C$ and $\omega_E$ denote the Euclidean metric. If $\Cstar$ acts on $\C$ in the standard way, then we obtain an effective algebraic action of $\Cstarsqr$ on $\C\P^{1} \times \C$. The product metric $\omega_\std = \omega_{FS} + \omega_E$ on $\C\P^{1} \times \C$ is then a complete $T^{2}$-invariant shrinking gradient K\"ahler-Ricci soliton with respect to the holomorphic vector field $z \frac{\p}{\p z}$ (here we suppress the obvious pullbacks). As an application of the results of the previous sections, we show that, up to isometry, this is the unique shrinking gradient K\"ahler-Ricci soliton on $\CP^1 \times \C$ with bounded scalar curvature.

\begin{thmC*} \label{thm4-9}
	Any complete shrinking gradient K\"ahler-Ricci soliton $(g,X)$ on $M = \C\P^{1} \times \C$ with bounded scalar curvature is isometric to to the standard product metric $\omega_\std$.
\end{thmC*}

By the work of \cite{MW}, in real dimension four we know that the scalar curvature controls the full curvature tensor for shrinking solitons. In particular, it follows from \cite[Theorem 1.3]{MW} that any such $(g,X)$ as above has bounded Ricci curvature. Fix a background product coordinate system $([w_1:w_2],z)$ on $M \cong \CP^1 \times \C$ as above. In what follows, we will ignore the standard $\Cstarsqr$-action determined by this choice, but we will routinely make use of the corresponding projection onto the $\C$-factor, which we denote by $\pi:M \to \C$. Corollary \ref{cor1-3} then follows from Theorem \ref{thm1-2} as soon as we have the following lemma. 

 \begin{proposition} \label{prop4-9}
 	Let $(g,X)$ be any complete shrinking gradient K\"ahler-Ricci soliton on $M = \C\P^{1} \times \C$ with bounded scalar curvature, and let $T \subset \Cstarsqr$ be the real torus corresponding to the standard $\Cstarsqr$-action on $M$ with Lie algebra $\t$. Then there exists a holomorphic automorphism $\alpha$ of $M$ such that $J(\alpha^*X) \in \t$. 
 \end{proposition}

The proof of this proposition will take up the remainder of this section. Let $f$ denote the soliton potential so that the soliton vector field $X = \nabla_g f$.  As before (c.f. Lemma \ref{lem4-1}), we define $G^X$ to be the of the group holomorphic isometries of $(M, J, g)$ that commute with the flow of $X$, and we let $G^X_0$ be the connected component of the identity in $G^X$. Then $G^X_0$ is a compact Lie group by \cite[Lemma 5.12]{ConDerSun}. Clearly the flow of $JX$ defines a one-parameter subgroup in $G^X_0$, and so the closure in $G^{X}_0$ is a real torus $T^X$ of holomorphic isometries of $g$. Let $M_0$ denote the zero set of $X$. Since the scalar curvature is bounded, it follows from \cite[Lemma 2.26]{ConDerSun} that $M_0$ is a compact analytic subvariety of $M$, and hence is equal to a finite collection of points in $M$ and curves $L_z = \CP^1 \times \{z\} \subset  M$. Note that the fixed point set of $T^X$ is equal to $M_0$. By Lemma \ref{lem4-1}, there exists a complexification $T^{X}_\C \subset \text{Aut}^X$ of $T^X$, which is a complex torus with $\dim_\C T^{X}_\C = \dim_\R T^X$. In what follows we will need to treat the the two possible cases, $\dim_\R T^X = 1$ and $\dim_\R T^X = 2$, separately. For the moment, we make no distinction.

We first study $M_0$, making use of the fact that $f$ is a Morse-Bott function on $M$ \cite{Fra}. Since $M$ is K\"ahler we have moreover that the Morse indices of any critical point must be even. Since $M_0$ consists of the critical points of $f$,  we can write 

\begin{equation*}
	M_0 = M^{(0)} \cup M^{(2)} \cup M^{(4)},
\end{equation*}
where $M^{(i)}$ denotes the connected component with Morse index $i$. By \cite[Claim 2.30]{ConDerSun}, we know that $M^{(0)}$ is a nonempty, compact, and connected analytic subvariety of $M$, and therefore must either be equal to a single projective line $L_z$ or an isolated point. We begin with a construction which will be used throughout the rest of the section. 

\begin{claim} \label{claim4-10}
	Suppose that $x$ is a point in $M^{(2)} \cup M^{(4)}$. Then there exists a holomorphic map $R_x: \CP^1 \to M$ with $R_x(0) = x$ and $R_x(\infty) \in M_0$ defined by the negative gradient flow of $f$. Since $M$ is a trivial $\CP^1$-fibration, the image of $R_x$ must lie in the unique fiber $L_z$ of $\pi$ containing $x$. 
\end{claim}

\begin{proof}
 By \cite[Proposition 6]{Bryant} there exists a local holomorphic coordinate system $(z_1, z_2)$ centered at $x$ such that the holomorphic vector field $X^{1,0} = \frac{1}{2}(X - iJX)$ is given by 

\begin{equation} \label{eqnXlin}
	X^{1,0} = a_1 z_1 \frac{\p}{\p z_1} + a_2 z_2 \frac{\p}{\p z_2}
\end{equation}
for $a_1, a_2 \in \R$. By assumption, Hess$_g(f)$ has at least one negative eigenvalue at $x$, and therefore we can assume without loss of generality that $a_2 < 0$. Then $JX$ is tangential to the $z_2$-axis, and the flow of $JX$ here is given by regular periodic orbits. We fix any such nontrivial orbit $\theta: S^1 \to M$. If we let $\psi_t: M \to M$ denote the flow of $-X = -\nabla_g f$, then we define a holomorphic map $r: \Cstar \cong S^1 \times \R \to M$ by $r(s,t) = \psi_t(\theta(s))$. It follows immediately from the local form \eqref{eqnXlin} that $r$ extends to a holomorphic map $r: \C \to M$ with $r_x(0) = x$. Now $f$ is bounded from below and decreases along its negative gradient flow, and therefore $f$ is bounded along the image of $r_x$. Since $f$ is proper, this implies that the image of $r_x$ lies in the compact set $f^{-1}((-\infty, a])$, where $a = \sup f \circ r_x$. If $\pi: M \to \C$ denotes the projection onto the second factor of $M = \CP^{1} \times \C$, then $\pi \circ r_x: \C \to \C$ is therefore bounded and hence constant. Thus, $\pi \circ r_x(\C) = z$ for some fixed $z \in \C$, so that the image of $r_x$ lies in $L_z = \pi^{-1}(z)$. For each fixed $s \in S^{1}$, we have by \cite[Proposition 2.28]{ConDerSun} a well-defined limit $\lim_{t \to \infty} \psi_t(\theta(s))$, also lying in $M_{0}$. In this case, the limits must all coincide with the unique point $p = L_z \backslash r_x(\C)$. Thus, there is a well-defined holomorphic extension of $R_x: \CP^1 \to M$ of $r_x$ with $R_x(\infty) = p$.
\end{proof}

\subsubsection{Case 1: $M^{(0)}$ is an isolated point}

\begin{claim} \label{claim4-11}
	Let $y$ be any point in $M^{(2)} \cup M^{(4)}$. Let $R_y: \CP^1 \to M$ be a holomorphic map with $R_y(0) = y$ and $R_y(\infty) \in M_0$, which must exist by Claim \ref{claim4-10}. Then $R_y(\infty) \in M^{(0)}$.
\end{claim}

\begin{proof} 
Set $p = R_y(\infty)$, and assume without loss of generality that $z = 0$, so that the image of $R_y$ is the fiber $L_0 = \pi^{-1}(0)$ of $\pi$. If $p \in M^{(4)}$, then we choose coordinates centered at $p$ in which $X^{1,0}$ takes the form \eqref{eqnXlin}. This immediately yields a contradiction, since $p$ is defined as the forward limit point of a flow line of $-X$. If both $a_i$ are negative, then no forward flow of $-X$ near $p$ converges to $p$. Thus, either $p \in M^{(0)}$ or $p \in M^{(2)}$. Since $L_0$ is the image of the map $R_y$ defined by the flow of $(X, JX)$, it follows that $X$ is tangential to $L_0$. In particular, the restriction $X|_{L_0}$ is a well-defined holomorphic vector field on $L_0$ and does not vanish identically since the map $R_y$ is non-constant. It follows that $M_0 \cap L_{0}$ consists only of the isolated points $x$ and $p$, and that $p$ is the point in $L_{0}$ at which $f$ attains its minimum value among all points in $L_{0}$. Suppose that $p \in M^{(2)}$. Then by Claim \ref{claim4-10}, there is a holomorphic embedding $R_p: \CP^1 \to M$ with $r_p(0) = p$, defined by the negative gradient flow of $f$. Thus, once again, the image of $R_p$ must be equal to $L_0$. This is a contradiction, since $f$ decreases along its negative gradient flow and $f(p) = \min_{L_0} f$.
\end{proof}

\begin{claim} \label{claim4-12}
	If we assume that $M^{(0)} = \{p\}$, then $M_0$ lies in a fixed fiber $L_0$ of $\pi$, and consists precisely of the two isolated points $M_0 = \{x\} \cup \{p\}$ with $x \in M^{(2)}$.
\end{claim}

\begin{proof}
In this case we have from \cite{Bryant} that $M^{(2)} \cup M^{(4)}$ must indeed be nonempty or else $M \cong \C^{2}$, which is clearly a contradiction. Let $x \in M^{(2)} \cup M^{(4)}$ be one such point. By Claim \ref{claim4-11}, there is a map $R_{x}: \CP^{1} \to M$ with $R_{x}(0) = x$ and $R_{q}(\infty) = p \in M^{(0)}$. In particular, $\pi(x) = \pi(p)$. Suppose that there is another point $q \in M_0$ not equal to $p$ or $x$. Then again by Claim \ref{claim4-11} there is a map $R_{q}: \CP^{1} \to M$ with $R_{q}(0) = q$ and $R_{q}(\infty) = p \in M^{(0)}$. Thus $R_q(\CP^{1}) = L_{0}$, which means in particular that $q \in  L_{0}$. This is a contradiction, since $q \neq p$ and $q \neq x$, and a holomorphic vector field on $\CP^{1}$ which vanishes at three distinct points must vanish identically. Finally, we claim that the point $x \in M^{(2)}$. If not, then $x \in M^{(4)}$, and \emph{both} coefficients $a_i$ in the representation \eqref{eqnXlin} for $X$ centered at $x$ are negative. Thus, there is a \emph{distinct}  holomorphic curve $R'_x: \CP^1 \to M$ with $R'_x(0) = x$, intersecting $R_x(\CP^{1})$ transversely at $x$. This is impossible, so we obtain our contradiction. 
\end{proof}

In particular, we have shown that if $M^{(0)} = \{p\}$, then the fixed point set of $T^X$ is finite. If $T^X$ is two-dimensional, then $T^X$ together with the K\"ahler form $\omega$ of $g$ give $M$ the structure of a symplectic toric manifold. We are therefore in the setting of the previous sections, and we can deduce Proposition \ref{prop4-9} from the results there.

\begin{claim} \label{claim4-13}
	Suppose that $T^X$ is contained in a two-dimensional real torus $\T$ acting on $M$ by holomorphic isometries of $\omega$. Then there exists an equivariant biholomorphism $\alpha: M \to \CP^1 \times \C$, where $\CP^1 \times \C$ is endowed with the standard $\Cstarsqr$-action. 
\end{claim}

\begin{proof}
	As we have seen in Section 4, the fact that $\omega$ is the K\"ahler form of a complete shrinking gradient K\"ahler-Ricci soliton on $M$ implies automatically that the $\T$-action is Hamiltonian. Since $\dim_\C \T_\C = 2 = \dim_\C M$ and the fixed point set is finite, we can apply Lemma \ref{lem4-4} to deduce that the image of the moment map $\mu$ is a Delzant polyhedron $P$ in $\text{Lie}(\T)^*$. Then Lemma \ref{lem2-22} implies that there exists an equivariant biholomorphism $\alpha: (M,J) \to (M_P, J_P)$, where $(M_P, J_P, \omega_P)$ is the AK-toric manifold of Proposition \ref{prop2-14}. By Proposition \ref{prop2-15}, $M_P$ is equivariantly biholomorphic to the unique algebraic toric variety $\M_P$ associated to $P$. It follows that the underlying complex structure of $M_P$ is biholomorphic to $\CP^1 \times \C$. Since the topology of an algebraic toric variety is uniquely characterized by its fan (c.f. \cite[Chapter 12]{CLS}), the only algebraic toric variety with this property is $\CP^1 \times \C$ with the standard $\Cstarsqr$-action up to equivariant isomorphism. Thus, $\alpha$ is the required biholomorphsim $\alpha: M \to \CP^1 \times \C$. 
\end{proof}

In particular, if $T^X$ itself is two-dimensional and $M^{(0)} = \{p\}$, then $T^X$ itself satisfies the hypotheses of Claim \ref{claim4-13}, and we can simply take $\T = T^X$. In fact, even when $\dim_{\R} T^X = 1$, we can always find a two dimensional torus $\T$ satisfying the hypotheses of Claim \ref{claim4-13}.

\begin{claim} \label{claim4-14}
	If $M^{(0)} = \{p\}$, then there exists a two-dimensional torus $\T$ of biholomorphisms acting on $M$ such that $T^X \subset \T$. 
\end{claim}

\begin{proof}
	If $T^X$ is two-dimensional, then there is nothing to prove. Therefore, we can assume that $T^X_\C$ defines an action of $\Cstar$ on $M$. Recall that $\pi$ denotes the projection $\pi : M \to \C$ under a fixed identification $M \cong \CP^1 \times \C$. Let $\varpi:M \to \CP^1$ denote the other projection. Then the $(1,0)$ tangent bundle $T^{1,0}_M$ of $M$ splits holomorphically as $T^{1,0}_M \cong \varpi^*T^{1,0}_{\CP^1} \oplus \pi^*T^{1,0}_{\C} $. In particular, there exist holomorphic projection maps onto the subbundles $\varpi^*T^{1,0}_{\CP^1}$ and $\pi^*T^{1,0}_{\C}$ of $T^{1,0}_M$. We can therefore write $X^{1,0} = V^{1,0} + W^{1,0}$, where $V^{1,0}, W^{1,0}$ are holomorphic vector fields lying in $\varpi^*T^{1,0}_{\CP^1}$ and $\pi^*T^{1,0}_{\C}$ respectively. 
	
	Notice that the coordinate $z$ on $\C$ defines a global holomorphic coordinate on $M$. Since $T^{1,0}_{\C} $ is trivial, we can write the vector field $W^{1,0}=f \frac{\p}{\p z}$, where $f$ is a holomorphic function on M. Now since $X^{1,0}$ generates $\Cstar$-action on $M$, $W^{1,0}$ also generates a $\Cstar$-action on $\C$. In particular, $f = f(z)$ depends only on $z$. Now $X^{1,0}$ is tangential to $L_0$, this action fixes $0 \in \C$. Since the automorphism group of $\C$ consists of linear transformations, it follows that $f(z)$ is of the form $f(z) = kz$.

	 For each $z \in \C$, the restriction of $V^{1,0}$ to $L_z$ is a holomrophic vector field on $L_z \cong \CP^1$ which we denote by $V_z^{1,0}$. A nonzero holomorphic vector field on $\CP^1$ vanishes at two points with multiplicity, so that $V_z^{1,0}$ either vanishes identically or has zero set equal to a degree 2 divisor in $\CP^1$. Recall that $X^{1,0}$ is tangential to $L_0$, and so $V_0^{1,0}$ vanishes only at $M_0$, which consists of the two isolated points $\{x\}$ and $\{p\}$. Thus, by the continuity of the map $\C \to H^{0}(\CP^{1}, \O(2))$ given by $z \mapsto V^{1,0}_z$, the same is true for $V_z^{1,0}$ with $|z|$ sufficiently small. In particular, there exists a small neighborhood $\Delta \subset \C$ of $0$ such that the zero set of $V_z^{1,0}$ for $z \in \Delta$ consists of disjoint embedded discs $\Delta_p, \Delta_x \subset M$, centered at $p$ and $x$ respectively, each meeting a given fiber $L_z$ at a unique point. Let $\{p_z\} = \Delta_p \cap L_z$ and $\{x_z\} = \Delta_x \cap L_z$. Let $y_0 \in L_0$ be a point which does not lie in $M_0$, and let $\Phi_t$ denote the flow of $W^{1,0}$. Since $W^{1,0} = kz \frac{\p}{\p z}$, clearly there exists a point $y \in M - L_0$ such that the orbit $W$-orbit $\Phi_t(y)$ of $y$ under the flow of $W^{1,0}$ converges to $y_0$ as $t \to 0$. Let $C^y \subset M$ be the closure of the orbit $\Phi_t(y)$ and let $\Delta_y$ denote the intersection of $C^{y}$ with $\CP^1 \times \Delta$. Again since $W$ takes this special form, and perhaps after shrinking $\Delta$, we can choose $y_0$ such that $\Delta_y$ does not intersect $\Delta_p \cup \Delta_x$. We denote the unique point of $\Delta_y \cap L_z$ by $y_z$. Then there is a unique automorphism $A_z \in \text{PGL}(2,\C)$ of $\CP^1$ such that $A_z(x_z) = \infty$, $A_z(y_z) = 1$, and $A_z(p_z) = 0$. Then we define an automorphism $\alpha_1: \CP^1 \times \Delta \to \CP^1 \times \Delta$ by setting $\alpha_1(\ell, z) = \left(A_z^{-1}(\ell), z\right)$. After changing coordinates on $\CP^1 \times \Delta$ by $\alpha_1$, we can assume that we have a homogeneous coordinate system $[w_1:w_2]$ on $\CP^1$ in which the vector field $V^{1,0}_z$ vanishes at the points $\{0\}$ and $\{\infty\}$. Up to scale, there is a unique holomorphic vector field on $\CP^1$ vanishing at two given points. If we set $w = \frac{w_1}{w_2}$, it follows then that $V^{1,0} = h(z) w \frac{\p}{\p w}$, where $h(z)$ is a holomorphic function only on $\Delta$ (notice that, although it is defined with respect to a coordinate system, $w \frac{\p}{\p w}$ is in fact a global holomorphic vector field on $\CP^1$).
	 
	  Now, $X^{1,0}$ generates a $\Cstar$-action on $M$, and moreover each orbit of this action intersects the neighborhood $\CP^{1} \times \Delta$ of $L_0$. Therefore, we can use the flow of $X^{1,0}$ itself to extend this local description. In particular, there is a global holomorphic extension $\alpha: M \to M$ of $\alpha_1$ inducing a change of coordinates on $M$ in which $X^{1,0}$ takes the form $X^{1,0} = h(z) w \frac{\p}{\p w} + kz \frac{\p}{\p z}$, where $w = \frac{w_1}{w_2}$ with respect to the homogeneous coordinates $[w_1 : w_2]$ on $\CP^1$ and now $h(z)$ is an entire holomorphic function on $\C$. Set $Y^{1,0} = w \frac{\p}{\p w}$. Then clearly $Y = \text{Re}(Y^{1,0})$ is complete and $[X,Y] = 0$. Furthermore, the flow of $(Y, JY)$ generates a $\Cstar$-action on $M$, which in these coordinates is just the standard action on $\CP^1$ on each fiber of $\pi$. Then $\Cstarsqr$ acts on $M$ via $X, JX, Y, JY$, and therefore we can take $\T$ to be the underlying real torus of this action.
\end{proof}

\subsubsection{Case 2: $M^{(0)}$ is a fiber of $\pi$}

\begin{claim} \label{claim4-15}
	Suppose that $M^{(0)}$ is a fiber of $\pi$, and so without loss of generality we may assume that $M^{(0)} = L_0$. Then both $M^{(2)}$ and $M^{(4)}$ must be empty.
\end{claim}

\begin{proof}
 Indeed, suppose that there exists a point $q \in M^{(2)}$. Let $z = \pi(q)$ so that $q \in L_z$. By assumption, $z \neq 0$. By Claim \ref{claim4-10}, there is a holomorphic embedding $R_q: \CP^{1} \to M$ defined by flowing along $(-X, -JX)$ with the property that $R_q(0) = q$ and $R_q(\CP^1) = L_z$. Set $q' = R_q(\infty)$, then it follows that the tangential component $V_{z}^{1,0}$ of $X^{1,0}$ to $L_z$ vanishes precisely at the two points $q, q'$ and that $q'$ is the point at which $f$ achieves $\min_{L_z}f$. In particular, $q'$ cannot lie in $M^{(0)} = L_0$, which means that $q' \in M^{(2)}$. But then we run the same argument at $q'$ to obtain a contradiction. Therefore $M^{(2)}$ must be empty. The case $q \in M^{(4)}$ is similar. Alternatively, one can see that $M^{(4)}$ is empty directly by an argument similar to the one in the proof of Claim \ref{claim4-12}.
 \end{proof}
 
 \begin{claim} \label{4-16}
 	If $M^{(0)} = L_0$, then $T^{X}$ is necessarily one-dimensional.
 \end{claim}
 
 \begin{proof}
 	Let $y \in M^{(0)}$. Choose coordinates $(z_1, z_2)$ in a neighborhood $U_y$ centered at $y$ such that $X^{1,0}$ takes the form \eqref{eqnXlin}. Since $y \in M^{(0)}$, we have that $a_1,a_2$ are both nonnegative. If $a_1$ and $a_2$ are both strictly positive, then it follows that every point $y' \in U_y$ lies on an orbit which converges as $t \to 0$ to $y$. Since $M^{(0)} = L_0$, we can choose a point $y' \in M^{(0)} \cap U_y$. Then $\Phi_t(y') \to y$ as $t \to 0$, which contradicts the fact that $X$ vanishes identically on $M_0$. Therefore, we may assume without loss of generality that $a_1 = 0$ and $a_2 > 0$. In particular, $L_0 \cap U_y$ is given by the $z_1$-axis and indeed all of the orbits of $(X, JX)$ in these coordinates are given by the affine lines $z_2 = const$. 
 	
 	 If $T^X_\C$ is two-dimensional, then as we have seen at the beginning of Section 2.2 there exists an orbit of $T^X_\C$ which is open and dense in $M$. The flow of $ JX$ determines by assumption a dense subgroup in $T^{X}$, and therefore there must be some point $q \in M$ such that the flow of $(X , JX)$ from $q$ is dense in $M$, and in particular is dense in $U_y$. But as we have seen, for sufficiently small $t$, the $\Phi$-orbit of any point in $U_y$ lies on a unique complex submanifold of $U_y$, the line $z_2 = const$. If the orbit $\Phi_t(q)$ is dense in $U_y$, pick two points $q_1, q_2$ such that $z_2(q_1) \neq z_2(q_2)$ and such that $q_2 = \Phi_{t^*}(q_1)$. By ensuring $q_2$ is close enough to the $z_1$-axis, we can futher assume that $|t^*|< 1$. By the local form \eqref{eqnXlin} we can see that the orbit of any point in $U_y$ of the punctured unit disc $\D^* \subset \Cstar$ is contained in $U_y$. In particular it follows that $z_1(q_2) = z_1(q_1)$, a contradiction. 
 \end{proof}

\begin{claim} \label{claim4-17}
	Let $p,q \in M - M^{(0)}$, and let $\Phi: \Cstar \times M \to M$ denote the complex flow of $(X, JX)$. If $\lim_{t \to 0} \Phi_t(p) = \lim_{t \to 0} \Phi_t(q) \in M^{(0)}$, then $q = \Phi_t(p)$ for some $t \in \Cstar$, i.e. $p$ and $q$ lie on the same orbit. 
	\end{claim}
	
	\begin{proof}
		This follows again from the local form \eqref{eqnXlin}. Since $M^{(i)}$ are empty for $i \neq 0$ by Claim \ref{claim4-15}, it must be that $\lim_{t \to 0} \Phi_t(p) \in M^{(0)}$ for all $p \in M$. Now suppose that $p, q \in M$ with $\lim_{t \to 0} \Phi_t(p) = \lim_{t \to 0} \Phi_t(q) = y \in M^{(0)}$. As we have seen, we can choose coordinates $(z_1, z_2)$ near $y$ in which $X^{1,0}$ takes the form \eqref{eqnXlin} where $a_1 = 0$ and $a_2 > 0$. It follows then that for sufficiently small $\varepsilon$, that both $\Phi_t(p)$ and $\Phi_t(q)$ lie on the line $z_2 = 0$ if $|t| < \varepsilon$. Thus the orbits from $p$ and from $q$ intersect, and are thereby equal. 
\end{proof}

We can now treat the final case that may arise. Together with Claims \ref{claim4-13} and \ref{claim4-14}, this completes the proof of Proposition \ref{prop4-9}.

\begin{claim} \label{claim4-18}
	If $M^{(0)} = L_0$, then there exists an equivariant biholomorphism $\alpha: M \to \CP^1 \times \C$, where $\CP^1 \times \C$ is endowed with the product $\Cstar$-action determined by the trivial action on $\CP^1$ and the standard one on $\C$. In particular, under the identification determined by $\alpha$, we have that $JX$ lies in the Lie algebra $\t$ of the standard $T^2$-action on $\CP^1 \times \C$. 
\end{claim}

\begin{proof}
From the proof of Claim \ref{4-16}, we know that $X^{1,0}$ satisfies $a_1 = 0, a_2 > 0$ with respect to the local form \eqref{eqnXlin}.  From this it is clear that the composition of any orbit $O_q: \Cstar \hookrightarrow M$ of $X^{1,0}$ with the projection $\pi:M \to \C$ defines a surjective map $\Cstar \to \Cstar$. In particular, if we let $\beta = a_2^{-1}$, then the orbits of $(\beta X, J(\beta X))$ intersect each fiber of $\pi$ precisely once. Now choose any fiber $L_z \cong \CP^{1}$ of $\pi$ in $M$ which is not equal to $L_0$, and let $\Phi^\beta$ denote the flow of $(\beta X, J (\beta X))$. We define a map $\alpha: M \to \CP^{1} \times \Cstar$ by the formula 
		\begin{equation*}
			\alpha(p) = (\Phi^\beta_{t^{-1}}(p), t),
		\end{equation*}
where $t \in \Cstar$ is the unique point such that $\Phi^\beta_{t^{-1}}(p) \in L_z$. By the previous claim, this extends to a biholomorphism $\alpha: M \to \CP^{1} \times \C$ such that $\alpha_*X^{1,0} = a_2 z \frac{\p}{\p z}$. 
\end{proof}

\section{Discussion} 

We pose some open questions related to the work here. For the most part, these problems have appeared in \cite{ConDerSun}. We reproduce them here, partially because they take on a slightly different light in the toric setting, and partially because they may simply be easier to prove in this context. 

\begin{enumerate}
	\item Is the assumption in Theorem \ref{thm1-2} that the Ricci curvature is bounded necessary? More specifically, suppose that $(M,J)$ admits an effective and holomorphic action of the real torus $T^{n}$. Given a complete shrinking gradient K\"ahler-Ricci soliton $(g,X)$ on $M$, does there exists a complexification of the $T^{n}$-action? We use the bound on the Ricci curvature to apply the work of \cite{ConDerSun} to show that there exists a complexification if the soliton vector field satisfies $JX \in \t$. Alternatively one could attempt to do away with the dependence on the full $\Cstarn$-action and corresponding dense complex coordinate chart. One can still interpret equation \eqref{gkrs} as an equation for the complex structure $J$ and produce a symplectic potential $u$ as in \cite{ACGT}. Our approach falls short at this stage, since we lack a method to determine good properties of the relevant functionals that appear in Section 3. 
	\item Suppose that $M$ is a toric manifold and $(g,X)$ is a complete shrinking gradient K\"ahler-Ricci soliton on $M$. Does there always exist an automorphism $\alpha$ of $M$ such that $\alpha^{*}g$ is invariant under the action of the real torus $T^{n}$? If we assume in addition that $g$ has bounded Ricci curvature, this is equivalent to the existence of an automorphism $\alpha$ such that $J\alpha^{*}X \in \t$. If so, then Theorem \ref{thm1-1} (resp.~Theorem \ref{thm1-2}) implies that $(g,X)$ is the unique complete shrinking gradient K\"ahler-Ricci soliton on $M$ (resp.~with bounded Ricci curvature). As it stands, we know little about the existence and uniqueness of shrinking solitons on $M$ without these hypotheses. We establish this in the special case that $M = \CP^{1} \times \C$ in Proposition \ref{prop4-9}, and Conlon-Deruelle-Sun show this for $M$ equal to $\C^{n}$ or the total space of the line bundle $\O(-k) \to \CP^{n-1}$ for $0 < k < n$ \cite[Theorem 5.20]{ConDerSun}.
 	\item Related to the previous question, suppose that $M$ is an arbitrary non-compact K\"ahler manifold and $X$ is a fixed holomorphic vector field. Is there at most one complete shrinking gradient K\"ahler-Ricci soliton $g$ on $M$ with $X$ as its soliton vector field? What if $g$ has bounded Ricci curvature? Moreover, is there at most one vector field $X$ on $M$ admitting a shrinking gradient K\"ahler-Ricci soliton? This is established by Tian-Zhu \cite{TZ2} for compact manifolds and by Conlon-Deruelle-Sun \cite{ConDerSun} for non-compact manifolds among all $Y$ such that $JY$ lie in the Lie algebra of a fixed real torus acting on $M$, with the estra assumption that the Ricci curvatureis bounded. We recover this result in Theorem \ref{thm4-7} in the toric setting.
	\item In this paper we work exclusively on smooth spaces $M$ to avoid technical complications. In the compact setting there has also been much interest surrounding weak K\"ahler-Einstein metrics and K\"ahler-Ricci solitons on singular spaces. Many of the techniques in this paper are adapted from the paper of Berman-Berndtsson \cite{BB}, in which such objects are of primary interest. Can the results here be generalized along the lines of \cite{BB} to include similar results for weak K\"ahler-Ricci solitons on non-compact \emph{singular} toric varieties? 
	 
\end{enumerate}

\bibliography{references}
\bibliographystyle{abbrv}

\end{document}